\documentclass{amsart}

\usepackage{amsmath,amssymb}
\usepackage{enumerate}
\usepackage{amsthm}

\usepackage{amssymb}
\usepackage{indentfirst}
\usepackage{comment}
\usepackage{amsthm}
\usepackage{amsmath}
\usepackage[final]{graphicx}
\usepackage{amsmath,amssymb}
\usepackage{float}
\usepackage{txfonts}
\usepackage{mathtools}

\usepackage{tikz}
\usetikzlibrary{cd}
\usetikzlibrary{quotes,angles,positioning,topaths,calc}

\title{A decomposition formula for the Bartholdi zeta function of a hypergraph covering}
\author{Kosei Watanabe}
\address{Graduate School of Mathematics, Nagoya University, Chikusa-ku, Furo-cho, Nagoya, 464-8602,  Japan}
\email{watanabe.kosei.g8@s.mail.nagoya-u.ac.jp}

\subjclass[2020]{}

\date{November 30, 2025}

\newtheorem{thm}{Theorem}[section]
{\theoremstyle{definition}\newtheorem{Def}[thm]{Definition}}
\newtheorem{prop}[thm]{Proposition}
\newtheorem{lem}[thm]{Lemma}
\theoremstyle{definition}

\newtheorem{rem}[thm]{Remark}

\newtheorem{ex}[thm]{Example}

\numberwithin{equation}{section}

\usepackage{latexsym}

\newcommand{\NN}{{\mathbb{N}}}
\newcommand{\ZZ}{{\mathbb{Z}}}

\newcommand{\CC}{{\mathbb{C}}}

\begin{document}
\bibliographystyle{amsalpha+}
\maketitle

%%%%%%%%%%%%%%%%%%%%%%%%%%%%%%%%%%%%%%%%%%%%%%%%%%%%%%%%%%%%%%%%%%%%%%

\begin{abstract}
It is shown by Mizuno and Sato that the Bartholdi zeta function
of a covering graph is decomposed as a product of 
 Bartholdi zeta functions of a base graph that are associated with representations.
In this paper, we extend their result to the case
 of a hypergraph covering.

\end{abstract}

\section{Introduction}
A $p$-adic analogue of the Selberg zeta function, 
which is now known as the 
Ihara zeta function, was introduced in \cite{Ihara}.
It is presented as an infinite product and is shown to be the reciprocal of a polynomial
(\cite{Ihara}).
Serre \cite{Serre1} pointed out
that this zeta function is regarded as that of regular graphs, 
and Bass \cite{Bass} extended this zeta function to the function of graphs. 
The Ihara zeta function of graph $G$ is given by

\begin{equation*}
    \zeta(G,t)=\prod_{[C]} \left(1-t^{|C|}\right)^{-1},
\end{equation*}
where the product runs over all equivalence classes $[C]$ of 
prime and reduced cycles $C$, and $|C|$ is the length of cycle $C$.

One line of research that has attracted a lot of interest is 
the investigation of decomposition formulae for graph coverings. 
A decomposition formula for the Ihara zeta function
 was firstly discussed in \cite{Stark}, where
 they showed the Ihara zeta function of 
 an unramified Galois covering of a graph is divisible by 
 that 
of the base graph.
Later, Mizuno and Sato extended the result of an unramified Galois covering to
a covering using ordinary voltage assignment in \cite{Mizuno00}.
Saito and Sato showed the decomposition formula 
for the Ihara zeta function of a 
hypergraph of a covering using ordinary voltage assignment in 
\cite{Sato13}.
On the other hand, 
Li and Hou gave a decomposition formula for the Ihara zeta function 
of a hypergraph of a permutation voltage
assignment in \cite{Li}.
Their result is as follows:
\begin{thm}[\cite{Li}]\label{thm11}
        We assume the following:
\begin{itemize}
    \item Let $H$ be a connected, finite hypergraph without loops, where
    each hypervertex belongs to at least two hyperedges (cf. Definition \ref{def:hypergraph}),
    \item Let $B_H$ be a bipartite graph associated with $H$, with
    $n= |V(B_H) | $ and $  m=|E(B_H)|$ (cf. Definition \ref{def:bipartite}),
    \item Let $\phi: E(R(B_H)) \to \mathcal{S}_k$ be a permutation voltage assignment of $B_H$ (cf. Definition \ref{def:pva}),
    \item Let $\Gamma = \langle \phi(e) \mid e \in R(B_H) \rangle$ be a subgroup of $\mathcal{S}_k$,
    \item Let $\Bar{H}$ be a $k$-fold covering of $H$ corresponding to $\phi$ and $\Gamma$ (cf. Definition \ref{def:hypercov}, Lemma \ref{lem36}),
    \item Let $\mathbb{P}$ be a permutation representation of $\Gamma$
    (cf. Definition \ref{def:rep}),
    \item Let $\rho_1=I, \ \rho_2,\dots,\rho_s$ be all the
    inequivalent irreducible representations of $\Gamma$,
    where $\rho_1$ denotes the trivial representation,
    \item Let $m_1,\dots,m_s$ be the multiplicities of irreducible representations
    $\rho_1,\dots,\rho_s$ in $\mathbb{P}$
    (cf. Lemma \ref{lem34}),
    \item Let $f_1,\dots,f_s$ be the degree of irreducible representations
    $\rho_1,\dots,\rho_s$.
\end{itemize}
Then the following decomposition formula holds:
        \begin{equation*}
            \zeta(\Bar{H},t)^{-1}=
            \zeta(H,t)^{-m_1}
            (1-t)^{(k-m_1)(m-n)}
            \prod_{i=2}^s M_i^{m_i},
        \end{equation*}
        where 
        \begin{equation*}
         M_i=\det \left(I_{f_in}-\sqrt{t}\sum_{g \in \Gamma}(\rho_i(g) \otimes A(B_H)_g)+
         tf_i \circ \left(D(B_H)-I_n\right)\right).
        \end{equation*}
    \end{thm}

Similarly,
the Bartholdi zeta function of a graph $G$, introduced in \cite{Bartholdi},
is a function of two variables defined by

\begin{equation*}
    \zeta(G,u,t)=\prod_{[C]} \left(1-u^{\mathrm{cbc}(C)}t^{|C|}\right)^{-1},
\end{equation*}
where the product runs over all equivalence classes $[C]$ of prime cycles $C$,
$|C|$ is the length of cycle $C$,
and $\mathrm{cbc}(C)$ is the cyclic bump count of a cycle $C$.

Like the Ihara zeta function, a decomposition formula 
for the Bartholdi zeta function has attracted a lot of interest.
A decomposition formula for the Bartholdi zeta function 
was studied 
by Mizuno and Sato 
\cite{Mizuno}.
Saito and Sato presented the decomposition formula for both the Ihara zeta function and the 
Bartholdi zeta function of a 
hypergraph of a covering using ordinary voltage assignment in 
\cite{Sato13}.

In general, the permutation voltage assignment is not a regular covering (see \cite{NS22}), which significantly distinguishes it from the ordinary voltage assignment. This non-regularity implies that the assignment carries less structural information regarding the group action, and thus, it is non-trivial to determine whether the decomposition formula, known for ordinary voltage assignments, remain valid. 
Despite this challenge, we successfully present
 this decomposition formula,
extending the result of \cite{Li} to the Bartholdi zeta function 
for a hypergraph covering using a permutation voltage assignment.

\begin{thm}[Theorem \ref{thm53}]\label{thm12}
Assume the same conditions as in Theorem \ref{thm11}
and all $i$ with $m_i>0$, $\rho_i$ is a unitary representation.
Then the following decomposition formula holds:
        \begin{equation*}
            \zeta(\Bar{H},u,t)
            =\prod_{i=1}^s\zeta(H,\rho_i,\phi,u,t)^{m_i},
        \end{equation*}
        where $\zeta(H,\rho_i,\phi,t)$ is the
        Bartholdi $L$-function of the hypergraph $H$ (cf. Definition \ref{def:L-funct.}).
    \end{thm}

In $\S2$, we prepare several basic  definitions such as hypergraphs, Bartholdi zeta functions,
and related topics.
In $\S3$, we set up some notation of hypergraph covering to state and prove our main theorem in $\S4$.
For the reader’s convenience, we provide complete proofs of all
statements required in this paper so that it is as self-contained as
possible.

\section{Hypergraphs}

In this section, we introduce basic definitions and notation 
related to the zeta functions of hypergraphs.

We regard a (unoriented) graph $G$ as a pair $(V(G),E(G))$ of sets where
the set $V(G)$ is the set of vertices of graph $G$ and 
the set $E(G)$ is the set of unoriented edges of graph $G$.

\begin{Def}
    A \textit{finite} and \textit{simple} graph $G$ 
    is defined as a pair of finite sets $(V(G),E(G))$ where
the set $E(G)$ is a set of unordered pairs of distinct vertices in $V(G)$.
\end{Def}

\begin{Def}
    Let $G$ be a graph. A vertex $v \in V(G)$ is 
    \textit{adjacent} to
    a vertex $u \in V(G)$ if there is an edge $e \in E(G)$ such that 
    $e=\{u,v\}$.
\end{Def}

\begin{rem}
    An edge of its cardinality $1$ is called a loop.
\end{rem}

We follow the definition of a hypergraph as given in \cite{Storm}.

\begin{Def}\label{def:hypergraph}
 A \textit{hypergraph} $H=(V(H),E(H))$ is 
 a pair of a set $V(H)$ of \textit{hypervertices} and 
 a set $E(H)$ of \textit{hyperedges}, such that the union of all hyperedges is $V(H)$.
 The set $E(H)$ is a collection of non-empty subsets of $V(H)$.
 Here, the hypergraph is finite.
 A hypervertex $v$ is \textit{incident} to a hyperedge $e$ if $v \in e$.    
\end{Def}

\begin{rem}
Let $H$ be a hypergraph.
    If any element of $E(H)$ has the cardinality two, then the hypergraph $H$ is 
    especially called the finite and simple graph.
\end{rem}

We provide the definition of a path in hypergraphs.

\begin{Def}
Let $H$ be a hypergraph.
\begin{itemize}
    \item  A \textit{path} $P$ of length $n$ in $H$ is a sequence 
 $P=(v_1,e_1,v_2,e_2,\dots,e_n,v_{n+1})$
 of $n+1$ hypervertices and $n$ hyperedges such that $v_i \in V(H), \ e_j \in E(H), $ and $ v_i,v_{i+1} \in e_i$ for $i=1,\dots,n$.
 We set $|P|=n$. The path $P$ is called a \textit{closed path} or \textit{cycle} if $v_1=v_{n+1}$.  
 \item The hypergraph $H$ is called \textit{connected}
 if there is a path between any two hypervertices.
\end{itemize}
  
\end{Def}

\begin{Def}
Let $H$ be a hypergraph.
 A path $P=(v_1,e_1,v_2,e_2,\dots,e_n,v_{n+1})$ in $H$
 has a \textit{backtracking} (a.k.a. \textit{bump}) at $e$ or $v$ if 
 % there is a 
 it contains a 
 subsequence of $P$ of the form $(e,v,e)$ or $(v,e,v)$,
 where $e \in E(H), \ v \in V(H)$.
 Furthermore, the \textit{cyclic bump count} $\mathrm{cbc}(C)$ of a cycle 
 $C=(v_1,e_1,v_2,e_2,\dots,e_n,v_1)$ is the sum of the following:
 \begin{equation*}
            \mathrm{cbc}(C)=| \ \{ i=1,\dots,n  \mid v_i=v_{i+1} \} \ |
            + | \ \{ i=1,\dots,n \mid e_i=e_{i+1} \} \ |,
 \end{equation*}
 where $v_{n+1}=v_1$ and $e_{n+1}=e_1$.    
\end{Def}

 We introduce an \textit{equivalence relation} between cycles.

 \begin{Def}
 Let $H$ be a hypergraph.
 \begin{enumerate}[(1)]
\item Two cycles $C_1=(v_1,e_1,v_2,\dots,e_m,v_1)$ and 
 $C_2=(w_1,f_1,w_2,\dots,f_m,w_1)$ in $H$ are called \textit{equivalent} if 
 there exists some $k$ such that 
 $w_j=v_{j+k \pmod m}, f_j=e_{j+k \pmod m }$ for all $j$.  
 % If $j+k >m$, then we subtract $m$ from $j+k$.
 We denote by $[C]$ the \textit{equivalence class} which contains the cycle $C$.
 \item Let $B^r \ ( r \ge 2)$ be the cycle obtained by going $r$ times around 
 the cycle $B$.
 Such a cycle is called a \textit{multiple} of $B$.  
 \item A cycle $C$ is \textit{prime} if it is not a multiple of any strictly 
smaller cycle.
 \item A cycle $C$ is \textit{reduced} if $\mathrm{cbc} (C)=0$.
 % $C^2$ has no backtrackings.
 \end{enumerate}
 \end{Def}

We define the zeta function of the graph.

\begin{Def}[\cite{Storm}] Let $H$ be a finite and connected hypergraph such that
 every hypervertex is in at least two hyperedges.
 The \textit{Ihara zeta function} of the hypergraph $H$ is defined by
 \begin{equation*}
     \zeta(H,t)=\prod_{[C]} \left(1-t^{|C|}\right)^{-1},
 \end{equation*}
 where $[C]$ runs over all equivalence classes of prime and reduced cycles of $H$,
 and $t \in \CC$ with $|t|$ sufficiently small for this function to
  converge.    
\end{Def}

\begin{Def}[\cite{Sato}] Let $H$ be a finite and connected hypergraph.
 The \textit{Bartholdi zeta function} of $H$ is defined by
 \begin{equation*}
    \zeta(H,u,t)=\prod_{[C]}\left(1-u^{\mathrm{cbc}(C)}t^{|C|}\right)^{-1},
 \end{equation*}
 where $[C]$ runs over all equivalence classes of prime cycles of $H$,
  and $u,t \in \CC$ with $|u|,  |t|$ sufficiently small.    
\end{Def}

We define the bipartite graph $B_H$ associated with a hypergraph $H$,
which is essential for the determinant expression of the zeta function 
given in Theorem \ref{thm213}.

\begin{Def}\label{def:bipartite}
A \textit{bipartite graph} $B_H$ associated with a hypergraph $H$ is the 
    (unoriented) graph defined as follows:
 $V(B_H)=V(H) \cup E(H)$ and $E(B_H)=\{ \{v,e\} \mid v \in V(H), e \in E(H), v \in e \}$. 
\end{Def} 

\begin{rem}\label{rem:bipar}
The bipartite graph $B_H$ is a simple graph because 
there is at most one edge 
between any hypervertex $v$ and 
any hyperedge $e$.
The bipartite graph $B_H$ is also a finite graph because
$H$ is a finite hypergraph.
This implies that 
$\#(V(H) \cup E(H)) = \#V(H)+\#E(H)<\infty$ and $\#E(B_H) < (\#V(H)) \times (\#E(H)) < \infty$.

\end{rem}

We prepare for the definition of a symmetric digraph.

\begin{Def}
Let $G$ be a simple graph.
Its \textit{symmetric digraph} $R(G)$
is a directed graph whose vertex set $V(R(G))$ is $V(G)$
and whose edge set $E(R(G))$ is $\left\{ (u,v),(v,u) \mid \{u,v\} \in E(G) \right\}$.
We call each element $e=(u,v) \in E(R(G))$ a \textit{directed edge} from the 
\textit{origin} vertex $u$ to the \textit{terminus} vertex $v$,
denoted by $o(e)=u$ and $t(e)=v$, respectively.
The \textit{inverse} of $e=(u,v)$ is the edge $e^{-1}=(v,u)$. 
\end{Def}

\begin{rem}\label{rem:24}
    In this paper, we assume that all graphs are finite.
    Let $r=\#E(G)$.
    For simplicity, we label the edges of $R(G)$ as
    $E(R(G))=\{e_1,e_2,\dots,e_r,e_{r+1},\dots,e_{2r} \}$, where
    $e_{i+r}=e_i^{-1}$ for $1 \le i \le r$.
\end{rem}

We now prepare some notation for paths and cycles in a symmetric
 digraph, which will be mainly utilized in Section $4$.

\begin{Def}\label{def:dipath}
    Let $R(G)$ be a symmetric digraph of finite,simple, and connected graph $G$.
    \begin{enumerate}[(1)]
        \item  A \textit{path} $P$ of length $n$ in $R(G)$ is a sequence 
        of $n$ directed edges 
 $P=(e_1,e_2,\dots,e_n)$
  such that 
 $t(e_i)=o(e_{i+1})$ for all $i \in \{1,\dots,n-1\}$.
 We set $|P|=n$. The path $P$ is called a \textit{closed path} or \textit{cycle} if $o(e_1)=t(e_{n})$.  
 \item Two cycles $C_1=(e_1,e_2,\dots,e_n)$ and 
 $C_2=(f_1,f_2,\dots,f_n)$ in $R(G)$ are called \textit{equivalent} if 
 there exists some $k$ such that 
 $f_i=e_{i+k \pmod n}$ for all $i$.  
 We denote by $[C]$ the \textit{equivalence class} which contains the cycle $C$.
 \item Let $B^r \ ( r \ge 2)$ be the cycle obtained by going $r$ times around 
 the cycle $B$.
 Such a cycle is called a \textit{multiple} of $B$.  
 \item A cycle $C$ is \textit{prime} if it is not a multiple of any strictly 
smaller cycle.
 \item A cycle $C=(e_1,\dots,e_n)$ is \textit{reduced} if $e_i \neq e_{i+1}^{-1}$ 
 for all $i \in \{1,\dots,n\}$.
    \end{enumerate}
\end{Def}

\begin{rem}\label{rem:corres}
For any graph $G$, there is a canonical bijection 
between the set of paths in $G$ and the set of paths in $R(G)$. This correspondence extends to cycles, prime cycles, and
reduced cycles.
\end{rem}

\begin{rem}

    Remark \ref{rem:corres} is used to establish the Hashimoto expression of zeta 
    functions of graphs. The generalization of this expression is given in 
    Proposition \ref{Prop:hashimoto}.
    Combining this with the correspondence of hypergraph and its bipartite graph, 
    we obtain the Hashimoto expression of zeta functions of a hypergraph.
\end{rem}

We give basic definitions from graph theory.

\begin{Def}\label{def:2.16}
Let the graph $G$ be finite and simple.
    \begin{itemize}
        \item Let $\#V(G)=n$. We define 
        the matrix $A(G) \in M_n(\ZZ)$
        whose $(i,j)$-th entry is $1$ if  
        the vertex $v_i$ is adjacent to the vertex $v_j$,
        and $0$ otherwise.
        Such a matrix is called the \textit{adjacency matrix} of the graph $G$.
        \item We define the \textit{degree} of the vertex $v_i$ 
        as the number of 
        the edges incident to vertex $v_i$.
        % $i$-th row of adjacency matrix.
        We denote the 
        degree of the vertex $v_i$ by $\deg (v_i)$.
        \item Let $\#V(G)=n$.
        We define the diagonal matrix $D(G) \in M_n(\ZZ)$
        whose $(i,i)$-th entry is $\deg (v_i)$ for all $i \in \{1,\dots,n\}$.
        Such a matrix is called the \textit{degree matrix} of the graph $G$.
    \end{itemize}
\end{Def}

A determinant expression of a Bartholdi zeta function of a hypergraph 
is given as follows.
\begin{thm}[\cite{Sato}]\label{thm213}
    Let $H$ be a finite and connected hypergraph such that every hypervertex is in at 
    least two hyperedges. Then
    \begin{equation*}
        \zeta(H,u,t)=\zeta(B_H,u,\sqrt{t})
        =(1-(1-u)^2t)^{-(m-n)}
        \det(I_n-\sqrt{t}A(B_H)+(1-u)t(D(B_H)-(1-u)I_n))^{-1},
    \end{equation*}
    where $n= |V(B_H) |$ and $ \ m=|E(B_H)|$.
\end{thm}
The Ihara zeta function is recovered when $u=0$.

\section{Hypergraph coverings over a finite hypergraph}

In this section, we discuss coverings of hypergraphs.

We begin with the definition of a covering graph.

\begin{Def}\label{def:covering}
Let $G$ be a simple graph.
     \begin{itemize}
         \item      We denote by $N(v)$  
     the \textit{neighborhood} of vertex $v$, 
     which is the set 
     $\{ u \in V(G) \mid \{u,v\} \in E(G) \}$
     % of vertices adjacent to $v$.
     \item   A graph $\Tilde{G}$ is called a \textit{graph covering} of $G$ with projection
$\pi: \Tilde{G} \to G$ if there is a surjective map
$\pi:V(\Tilde{G}) \to V(G)$ such that $\left.\pi\right|_{N(v')} :N(v') \to N(v)$
 is a bijection for all vertices $v \in V(G)$ and $v' \in \pi^{-1}(v)$.
 \item 
 A covering graph $\Tilde{G}$ is called a 
 \textit{k-fold covering} if $\pi$ is $k$-to-one.
 
     \end{itemize}
\end{Def}

We prepare for the definition of a hypergraph covering.
\begin{Def}\label{def:hypercov}
    Let $H$ be a hypergraph.
     \begin{itemize}
         \item      We denote by $N(v)$  
     the \textit{neighborhood} of vertex $v$, 
     which is the set 
     $\{ u \in V(H) \mid \{u,v\} \subset e \in E(H) \}$
     % of vertices adjacent to $v$.
     \item   A hypergraph $\Bar{H}$ is called a \textit{hypergraph covering} of $H$ with projection
$\pi: \Bar{H} \to H$ if there is a surjective map
$\pi:V(\bar{H}) \to V(H)$ such that $\left.\pi\right|_{N(\Bar{v})} :N(\Bar{v}) \to N(v)$
 is a bijection for all vertices $v \in V(H)$ and $\Bar{v} \in \pi^{-1}(v)$.
 \item 
 A covering hypergraph $\Bar{H}$ is called a 
 \textit{k-fold covering} if $\pi$ is $k$-to-one.
     \end{itemize}
\end{Def}

We construct a $k$-fold graph using a graph and a group.

\begin{Def}[\cite{Li}]\label{def:pva}
Let $G$ be a finite and simple graph.
Let $\mathcal{S}_k$ denote the symmetric group acting on the set $[k]=\{1,2,\dots,k\}$.
% be the \textit{symmetric group} which acts on $[k]=\{1,2,\dots,k\}$.
A map $\phi : E(R(G)) \to \mathcal{S}_k$ is called a
\textit{permutation voltage assignment} on $G$ if
 $\phi(e^{-1}) = \phi(e)^{-1}$ for all $e \in E(R(G))$.
 The pair $(G,\phi)$ is called a \textit{permutation voltage graph}.
\end{Def}

We define the derived graph, which is the graph associated with a permutation voltage assignment.
\begin{Def}[\cite{Li}]
 The \textit{derived graph} $G^{\phi}$ of the permutation voltage graph $(G,\phi)$
 is defined by:
 $V(G^{\phi})=V(G) \times [k]=\{ (v,i) \mid v \in V(G) , \ i \in [k]  \}$ and 
 $((u,i),(v,j)) \in E(R(G^{\phi}))$ if and only if there exists an edge $e=(u,v) \in E(R(G))$ such that  $i=\phi(e)j$.    
\end{Def}

We prepare the notation on representation of a subset of $\mathcal{S}_k$.

\begin{Def}\label{def:rep}
Let $\mathcal{S}_k$ denote the symmetric group acting on the set $[k]=\{1,2,\dots,k\}$.
Let $\mathcal{H}$ be a subgroup of $\mathcal{S}_k$.
\begin{itemize}
    \item 
A map $\rho : \mathcal{H} \to GL(k, \mathbb{C})$ is called a \textit{representation} of $\mathcal{H}$ if it is a group homomorphism.
\item Let $\mathbb{P}$ be the (left) permutation representation of 
$\mathcal{H}$ of degree $k$.
For each $g \in \mathcal{H}$, the $(i,j)$-th entry $p_{ij}^{(g)}$ of $\mathbb{P}(g) $ is defined as 
    \begin{equation*}
        p_{ij}^{(g)}=\delta_{i,g(j)},
    \end{equation*}
where $\delta$ is the \textit{Kronecker delta}.
\end{itemize}

\end{Def}

For square matrices $A_1,A_2,\dots,A_n$, let
\begin{equation*}
    A_1 \oplus \cdots \oplus A_n=\oplus_{i=1}^n A_i
    =\mathrm{diag} (A_1,A_2,\dots,A_n)
\end{equation*}
For simplicity, we let $n \circ A$ denote $\oplus_{i=1}^n A$.

The following lemma is from \cite{Serre}.

\begin{lem}[\cite{Serre}]\label{lem34}
Let $\mathcal{H}$ be a subgroup of $\mathcal{S}_k$. 
    Let $\rho_1=I_1,\rho_2,\dots,\rho_s$ be all the
    \textit{inequivalent irreducible representations} of $\mathcal{H}$.
    Let $f_i$ be the \textit{degree} of $\rho_i$ and 
    $m_i$ be the \textit{multiplicity} of $\rho_i$ in the 
    permutation representation $\mathbb{P}$.
    Then there exists an invertible matrix $M$ such that 
    for all $ g \in \mathcal{H}$, 
    \begin{equation*}
        M^{-1}\mathbb{P}(g) M=m_1 \circ I_1 \oplus m_2 \circ \rho_2(g)
        \oplus \cdots \oplus m_s \circ \rho_s(g).
    \end{equation*}
\end{lem}

\begin{rem}[\cite{Serre}]
 The multiplicity of the trivial representation $m_1$ is bigger than $0$.
\end{rem}

\begin{lem}[{\cite[Theorem 7]{Li}}] \label{lem36}
    Let $\Bar{H}$ be any $k$-fold hypergraph covering of the hypergraph $H$.
    Then there exists a unique permutation voltage assignment $\phi:E(R(B_H)) \to \mathcal{S}_k$
    % on the bipartite graph $B_H$ of $H$ 
    such that the hypergraph
    $H^{B_H^{\phi}}$, which is  
    associated with $B_H^{\phi}$(the derived bipartite graph),
    is isomorphic to $\Bar{H}$.
\end{lem}

The matrix 
$A(B_H)_g$ is defined as follows:

\begin{Def}\label{def:38}
Let $H$ be a finite and connected hypergraph.
Let a map $\phi:E(R(B_H)) \to \mathcal{S}_k$ be a permutation voltage assignment on $B_H$,
     and let $n=\# V(H), \ m=\#E(H)$, where $n+m=\#V(B_H)$.
Let $\Gamma $ be a group that is generated by $ \mathrm{Im} \ \phi$.
For any $g \in \Gamma$, the $(n+m) \times (n+m)$ matrix 
$A(B_H)_g=(a_{ij}^{(g)})$ is defined by
    \begin{equation*}
    a_{ij}^{(g)}=
        \begin{cases}
            1 &  \text{if} \  e=(v_i,v_j) \in E(R(B_H)),  \text{where} \ 
            v_i,v_j \in V(R(B_H))\ 
            \text{and}\  \phi(e)=g,\\
            0 & \text{otherwise.}
        \end{cases}
    \end{equation*}
\end{Def}

    The following equality holds:
    \begin{equation*}
        A(B_H)=\sum_{g \in \Gamma}A(B_H)_g,
    \end{equation*}
    where $A(B_H)$ is the adjacency matrix in Definition \ref{def:2.16}.

To state Lemma \ref{lem38}, we recall the definition of the Kronecker product.
\begin{Def}\label{def:kronecker}
For matrices $A=(a_{ij})$ and $B$, we denote the 
\textit{Kronecker product} as $A \otimes B$.
It is a block matrix where the $(i,j)$-th block is the product of 
$a_{ij}$ and the matrix $B$.      
\end{Def}

    \begin{lem}\label{lem38}
        Let $H$ be a finite hypergraph with no loops,
        where every hypervertex belongs to at least two 
        hyperedges.
        Let $\#V(H)=n $ and $ \#E(H)=m$.
        Let $B_H$ be the associated bipartite graph of $H$, and 
        $\phi:E(R(B_H)) \to \mathcal{S}_k$ be a
        permutation voltage assignment for $B_H$.
        Let the hypergraph $\Bar{H}$ be a $k$-fold hypergraph covering of $H$ corresponding to $\phi$ and $\Gamma=\langle \phi(e) \mid e \in R(B_H) \rangle$.
        Let $\mathbb{P}$ be the permutation representation of $\Gamma$.
        Then the following equality holds:
        \begin{equation*}
            A(B_{\Bar{H}})=\sum_{g \in \Gamma}(\mathbb{P}(g) \otimes A(B_H)_g).
        \end{equation*}
    \end{lem}

\begin{proof}
    By Lemma~\ref{lem36}, we have $A(B_{\Bar{H}}) = A(B_H^\phi)$.  
    The vertex set of $B_H^\phi$ can be ordered as 
    $(v_1^{(1)},\dots,v_n^{(1)},e_1^{(1)},\dots,e_m^{(1)},\dots,
    v_1^{(k)},\dots,v_n^{(k)},e_1^{(k)},\dots,e_m^{(k)})$,
    where $v_i \in V(H), \ e_j \in E(H)$.
    We represent both sides of the equation as $k \times k$ block matrices. 
    The matrix on the right hand side is given by:
    \begin{equation*}
        \sum_{g \in \Gamma}(\mathbb{P}(g) \otimes A(B_H)_g)
     =(A_{ij})_{1\le i,j \le k},
    \end{equation*}
    with
    $A_{ij}=\sum_{g \in \Gamma}p_{ij}^{(g)}A(B_H)_g \in M_{n+m}(\ZZ)$.
    The matrix on the left hand side, $A(B_H^\phi)$,
    can also be viewed as $k \times k$ block matrix 
    $(\Bar{A}_{ij})_{1\le i,j\le k}$. Each matrix  
    $\Bar{A}_{ij}$ is $(n+m) \times (n+m)$ matrix 
    representing the adjacency between 
    $(v_1^{(i)},\dots,v_n^{(i)},e_1^{(i)},\dots,e_m^{(i)})$
    and 
    $(v_1^{(j)},\dots,v_n^{(j)},e_1^{(j)},\dots,e_m^{(j)})$.
    It suffices to show that $\Bar{A}_{ij}=A_{ij}$ for all $1 \le i,j \le k$.
    We prove this by comparing their $(s,t)$-th entry for $1 \le s,t \le n+m$.
    Consider the $(s,t)$-th entry of $A_{ij}$,
    given by $\sum_{g \in \Gamma}p_{ij}^{(g)}a_{st}^{(g)}$.
    The entry $a_{st}^{(g)}=1$ if and only if there is 
    a directed edge $e=(v_s,v_t) \in E(R(B_H))$ with $\phi(e)=g$.
    The entry $p_{ij}^{(g)}=1$ if and only if $i=g(j)$.
    Both $a_{st}^{(g)}=1$ and $p_{ij}^{(g)}=1$ happen when
    $e=(v_s,v_t) \ \phi(e)=g $ and $i=g(j)$.
    Therefore, if there exists an element $g \in \Gamma$ such that 
    $a_{st}^{(g)}p_{ij}^{(g)}=1$, then it is unique since 
    it is the image of the map $\phi$.
    Then the sum $\sum_{g \in \Gamma}p_{ij}^{(g)}a_{st}^{(g)}$
    is $1$ if there exists an element $g \in \Gamma$ that 
    satisfies both conditions.
    Next, We consider the $(s,t)$-th entry of $\Bar{A}_{ij}$.
    This entry is $1$ if and only if there exists a directed edge 
    $(v_s^{(i)},v_t^{(j)})$.
    This happens when $e=(v_s,v_t) \in E(R(B_H))$ and 
    $i=\phi(e)j$ hold.
    Thus, the $(s,t)$-th entry of $\Bar{A}_{ij}$ is $1$
    if and only if there is an edge $e=(v_s,v_t)\in E(R(B_H))$ 
    with $g=\phi(e)$ such that $i=g(j)$.
    This condition is exactly the same as for the $(s,t)$-th entry of 
    $A_{ij}$ to be $1$.
    Therefore, their entries are equal, which completes the proof.

\end{proof}

\section{Main theorem}
In this section, we prove Theorem \ref{thm12},
which is stated as Theorem \ref{thm53} in this section.

    \begin{thm}\label{thm41}
        Let us adopt the same assumption as in Theorem \ref{thm11}.
        Then the Bartholdi zeta function of $\Bar{H}$ is decomposed as follows:
        \begin{equation*}
            \zeta(\Bar{H},u,t)^{-1}=
            \zeta(H,u,t)^{-m_1}
            (1-(1-u)^2t)^{(k-m_1)(m-n)}
            \prod_{i=2}^s M_i^{m_i},
        \end{equation*}
        where 
        \begin{equation*}
         M_i=\det \left(I_{f_in}-\sqrt{t}\sum_{g \in \Gamma}(\rho_i(g) \otimes A(B_H)_g)+
         (1-u)tf_i \circ \left(D(B_H)-(1-u)I_n\right)\right).
        \end{equation*}
    \end{thm}

\begin{proof}
The following argument follows the method introduced in \cite{Li}.
We begin with the following identity:
    \begin{equation*}
        D(B_{\Bar{H}})=I_k \otimes D(B_H).
    \end{equation*}
    By Lemma \ref{lem38}, we have the following equality:
        \begin{equation*}
            A(B_{\Bar{H}})=\sum_{g \in \Gamma}(\mathbb{P}(g) \otimes A(B_H)_g).
        \end{equation*}
    Therefore, by Theorem \ref{thm213}, we obtain
        \begin{align}\label{eq:det}
            \zeta(\Bar{H},u,t)^{-1}
            &=\zeta(B_{\Bar{H}},u,\sqrt{t})^{-1} \notag \\
            &=(1-(1-u)^2t)^{k(m-n)}
            \det(I_{kn}-\sqrt{t}A(B_{\Bar{H}})+(1-u)t(D(B_{\Bar{H}})-(1-u)I_{kn})) \notag \\
            &=(1-(1-u)^2t)^{k(m-n)} \cdot \notag \\
            &\det\left(I_{kn}-
            \sqrt{t}\left(\sum_{g \in \Gamma}(\mathbb{P}(g )\otimes A(B_H)_g)\right)
            +(1-u)t \left(I_k \otimes D(B_H)-(1-u)I_{kn} \right)\right).
        \end{align}
        Then we consider the expression
        \begin{equation*}
            I_{kn}-
            \sqrt{t}\left(\sum_{g \in \Gamma}(\mathbb{P}(g) \otimes A(B_H)_g)\right)
            +(1-u)t(I_k \otimes D(B_H)-(1-u)I_{kn}).
        \end{equation*}
        According to Lemma \ref{lem34}, 
        we can find an invertible matrix $M$.
        Multiplying the expression by $M^{-1} \otimes I_{n}$ from the left 
         and $M \otimes I_{n}$ from the right, we obtain:

\begin{align*}
    \left(M^{-1} \otimes I_{n} \right) \left(\sum_{g \in \Gamma}(\mathbb{P}(g) \otimes A(B_H)_g)\right)
    \left( M \otimes I_{n} \right)
    &=\sum_{g \in \Gamma}(M^{-1}\mathbb{P}(g) M)\otimes A(B_H)_g\\
    &=\sum_{g \in \Gamma} \left(\oplus_{i=1}^s m_i \circ \rho_i(g) \right) 
    \otimes A(B_H)_g\\
    &=\sum_{g \in \Gamma}
    \left( m_1 \circ \rho_1(g) \oplus \oplus_{i=2}^s m_i \circ \rho_i(g) \right) 
    \otimes A(B_H)_g\\
    &=\sum_{g \in \Gamma}
    \left(m_1 \oplus \oplus_{i=2}^s m_i \circ \rho_i(g) \right) 
    \otimes A(B_H)_g\\
    &=\sum_{g \in \Gamma}\left(m_1 \otimes A(B_H)_g \right)
    \oplus \left(\left(\oplus_{i=2}^s m_i \circ \rho_i(g) \right) 
    \otimes A(B_H)_g\right)\\
    &=\left(I_{m_1} \otimes A(B_H)\right) \oplus
    \left( \oplus_{i=2}^s {m_i \circ \left(\sum_{g \in \Gamma} (\rho_i(g) \otimes A(B_H)_g )\right)} \right).
\end{align*}

It also satisfies the following:
\begin{align*}
    \left(M^{-1} \otimes I_{n} \right) \left(I_k \otimes D(B_H)\right)
    \left( M \otimes I_{n} \right)
    &=I_k \otimes D(B_H).
\end{align*}

 In conclusion, the determinant in equation \eqref{eq:det} simplifies to the following:
\begin{align*}
    &\det \left(I_{m_1n}-\sqrt{t}\left(I_{m_1}\otimes A(B_H) \right)+
    (1-u)t (I_{m_1} \otimes D(B_H)-(1-u)I_{m_1n}) \right)\\
    &\cdot \prod_{i=2}^s
    \det \left(I_{m_if_in}-\sqrt{t}\left( m_i \circ \left(\sum_{g \in \Gamma} \rho_i(g) \otimes A(B_H)_g\right) \right)+
    (1-u)t \left(I_{m_if_i} \otimes D(B_H)-(1-u)I_{f_im_in}\right) \right)\\
    &=\det \left(I_{n}-\sqrt{t} A(B_H) +
    (1-u)t ( D(B_H)-(1-u)I_{n}) \right)^{m_1}\\
    &\cdot \prod_{i=2}^s
    \det \left(I_{f_in}-\sqrt{t} \left(\sum_{g \in \Gamma} \rho_i(g) \otimes A(B_H)_g\right) +
    (1-u)t \left(I_{f_i} \otimes D(B_H)-(1-u)I_{f_in}\right) \right)^{m_i}\\
    \intertext{By Theorem \ref{thm213},}
    &=\zeta(H,u,t)^{-m_1}(1-(1-u)^2t)^{-m_1(m-n)}
    \cdot \prod_{i=2}^s M_i^{m_i}.
\end{align*}

Combining all of the above, we obtain the following:
\begin{align*}
    \zeta(\Bar{H},u,t)^{-1}
    &=(1-(1-u)^2t)^{k(m-n)}\zeta(H,u,t)^{-m_1}
    (1-(1-u)^2t)^{-m_1(m-n)}\cdot \prod_{i=2}^s M_i^{m_i}\\
    &=\zeta(H,u,t)^{-m_1}(1-(1-u)^2t)^{(k-m_1)(m-n)}\cdot \prod_{i=2}^s M_i^{m_i}.
\end{align*}
\end{proof}

We now turn to the interpretation of the factors $M_i$ appearing in Theorem \ref{thm41}. 
In order to describe them in terms of $L$-functions, 
we work under the same assumption as in Theorem \ref{thm11}. 
More precisely, we define $L$-functions with respect to a representation $\rho$ and 
the permutation voltage assignment $\phi$ defined on $E(R(B_H))$. 
From this point on, we work with the symmetric digraph $R(B_H)$ rather than $B_H$. 
    This is justified by Remark \ref{rem:corres},
    which establishes a canonical bijection between paths 
    (and thus cycles, prime cycles, and reduced cycles) in $B_H$ and in $R(B_H)$.
Therefore, no information is lost in passing to $R(B_H)$.

For any sequence of directed edges $C=(e_1,\dots,e_s)$ and for any representation $\rho$, we set $\rho(\phi(C))=\rho(\phi(e_1))\cdots \rho(\phi(e_s))$.

We define the Bartholdi $L$-function as follows:

    \begin{Def}[\cite{Mizuno}]\label{def:L-funct.}
    Let us adopt the same assumption as in Theorem \ref{thm41},
    and let $\rho$ be a representation of $\Gamma$ with degree $l$.
        We define the \textit{Bartholdi $L$-function} of the hypergraph $H$
        with respect to $\rho$ and $\phi$ as follows:
        \begin{equation*}
            \zeta(H,\rho,\phi,u,t)
            =\zeta(B_H,\rho,\phi,u,\sqrt{t})=
            \prod_{[C]}\det\left(I_l -\rho(\phi(C))u^{cbc(C)}{\sqrt{t}}^{|C|} \right)^{-1},
        \end{equation*}
        where, $[C]$ runs over all equivalence classes of prime cycles of $R(B_H)$,
        and $u,t \in \CC$ with $|u| $ and $ |t|$ both sufficiently small.
    \end{Def}

The value of Bartholdi $L$-function is 
independent of the choice of the representatives
for the equivalence classes of prime cycles,
due to the following Lemma \ref{lem:cycleeq}.

\begin{lem}\label{lem:cycleeq}
Let us adopt the same assumption as in Definition \ref{def:L-funct.}.
Let $C$ be a cycle,
and let $\Tilde{C} \in [C]$, where $[C]$ is the equivalence class of $C$.
Then the following holds:
\begin{equation}\label{eq:lemCD}
    \det \left(I_{l}-\rho(\phi(C))u^{\mathrm{cbc}(C)}t^{|C|}\right)=
    \det \left(I_{l}-\rho(\phi(\Tilde{C}))u^{\mathrm{cbc}(\Tilde{C})}t^{|\Tilde{C}|}\right).
\end{equation}
\end{lem}
\begin{proof}
It suffices to consider the case $\Tilde{C} \neq C$, as the case $\Tilde{C}=C$ is trivial.
    Let $C=(e_{i_1},\dots,e_{i_k})$ be a cycle of length $k$.
    Since $\Tilde{C}$ is equivalent to $C$, $\Tilde{C}$ is a cyclic permutation of $C$.
    That is, $\Tilde{C}=(e_{i_s},\dots,e_{i_{s-1}})$ for some $s \in \{2,\dots,k\}$.
    Let $N=\rho(\phi(e_{i_1}))\rho(\phi(e_{i_2})) 
    \cdots \rho(\phi(e_{i_{s-1}}))$
    and we multiply the 
    left hand side of equation \eqref{eq:lemCD}
    by $\det (N^{-1})$ on the left 
    and by $\det (N)$ on the right.
    \begin{align*}
    \det \left(I_{l}-\rho(\phi(C))u^{\mathrm{cbc}(C)}t^{|C|}\right)
    &=\det (N^{-1})\det 
        \left(I_{l}-\rho(\phi(C))u^{\mathrm{cbc}(C)}t^{|C|}\right)\det (N)\\
        &=\det \left(N^{-1}I_lN-N^{-1}\rho(\phi(C))N
        u^{\mathrm{cbc}(C)}t^{|C|}\right)\\
        &=\det \left( I_l-
        \rho(\phi(e_{i_s})\cdots \rho(\phi(e_{i_k}))\rho(\phi(e_{i_1}))\cdots \rho(\phi(e_{i_{s-1}}))
        u^{\mathrm{cbc}(C)}t^{|C|}\right)\\
        &=\det \left( I_l-
        \rho(\phi(e_{i_s})\cdots \cdots \rho(\phi(e_{i_s-1}))
        u^{\mathrm{cbc}(C)}t^{|C|}\right)\\
        &=\det \left(I_{l}-\rho(\phi(\Tilde{C}))u^{\mathrm{cbc}(C)}t^{|C|}\right)
    \end{align*}
    Since $\Tilde{C}$ is a cyclic permutation of $C$, their length and 
    $\mathrm{cbc}$ values are equal.
    That is, $|C|=|\Tilde{C}|$ and 
    $\mathrm{cbc}(C)=\mathrm{cbc}(\Tilde{C})$.
    Therefore, the equality holds.
\end{proof}

We recall the notation of Lyndon words,
which provides a convenient way to represent the prime cycles of a graph.
Using this terminology, the Hashimoto expression for the $L$-function,
presented in Proposition \ref{Prop:hashimoto},
can be derived more easily.

\begin{Def}[\cite{Lothaire}]\label{def:lotha}
Let $A$ be a finite set with a total order,
which extends to a lexicographic order on the free monoid $A^*$ 
generated by $A$. Each element in $A^*$ is called a \textit{word}.
The empty word is denoted by $1$.
\begin{enumerate}[(1)]
    \item The word $u \in A^*$ has length $k$
    if $u=u_1\dots u_k$, where $u_i \in A$ for all $i\in \{1,\dots,k\}$.
    \item Two words $u,v \in A^*$ are \textit{conjugate} 
    if there exist two words $s,t \in A^*$ such that $u=st,\ v=ts$.
    This makes an equivalence relation on $A^*$.
    \item A word $u \in A^*$ is \textit{primitive}
    if $u \neq 1$ and 
    it is not a power of another word; that is 
    $u  \neq w^l$ for any word $ w \in A^* $ and for any integer $ l \ge 2$.
    \item  A word $w \in A^*$ is a
    \textit{Lyndon word} if it is primitive and 
    it is the lexicographically smallest word in its conjugacy class.
\end{enumerate}
\end{Def}

\begin{lem}[{\cite[Proposition 1.3.3]{Lothaire}}]\label{lem:prilo}
Let $A$ be a finite set and $A^*$ be the free monoid generated by $A$.
Then primitive words are conjugate to only primitive words.
\end{lem}

\begin{rem}\label{rem:Lyncom}
By Definition \ref{def:lotha} and Lemma \ref{lem:prilo},
the set of Lyndon words over the finite set $A$ 
is a complete set of representatives for 
the conjugacy classes of primitive words over $A$.    
\end{rem}

Under the same assumption as in Theorem \ref{thm11},
we can rephrase the equivalence relation 
on the sequences of the directed edges,
defined by cyclic permutation,
using the terminology of words
on their indices.
Two sequences of directed edges 
$C=(e_{i_1},\dots,e_{i_{k}})$ and 
$D=(e_{j_1},\dots,e_{j_{k}})$ are equivalent 
if and only if
their index words $i_1 \cdots i_{k}$ and $j_1 \cdots j_k$
are conjugate as the words over the finite set $\{1,\dots,2m\}$.
This is a direct consequence of the definition of conjugacy.

We prepare the following notation for the proof of Theorem \ref{thm53}.
Under the same assumption as in Theorem \ref{thm11},
we define the $2ml \times 2ml$ matrices matrices $B$ and $J$.
The entries of these matrices are $l \times l$ matrices.
\begin{align*}
    B&=(b_{\alpha \beta})_{1\le \alpha,\ \beta \le 2m}, \quad 
    b_{\alpha\beta} =
    \begin{cases}
        \rho(\phi(e_\alpha)) &\text{if} \   t(e_\alpha)=o(e_\beta), \ e_\alpha \neq e_\beta^{-1},\\
        0_l & \text{otherwise},
    \end{cases}\\
    J&=(j_{\alpha\beta})_{1\le \alpha, \ \beta  \le 2m},\quad 
    j_{\alpha\beta}=
    \begin{cases}
        \rho(\phi(e_\alpha)) & \text{if} \  e_\alpha = e_\beta^{-1},\\
        0_l & \text{otherwise}.
    \end{cases}
\end{align*}

\begin{lem}\label{lem:45}
Let us adopt the same assumption as in Theorem \ref{thm11}.
Let $M_1,\dots,M_{2m}$ be $2ml \times 2ml$ matrices such that, for each $i$, the rows indexed by $l(i-1)+1,\dots,li$ of $M_i$ coincide with the corresponding rows of $B+uJ$, while all other rows are zero.
Let $C=(e_{s_1},\dots, e_{s_k})$ be a sequence of directed edges.
Then we have 
\begin{equation*}
    \det \left(I_{2ml}-M_{s_1}\cdots M_{s_k}t^{|C|}\right)=
    \begin{cases}
        \det \left(I_{l}-\rho(\phi(C))u^{\mathrm{cbc}(C)}t^{|C|}\right),
        & \text{if} \  C \text{ is a cycle},\\
        1, & \text{otherwise.}
    \end{cases}
\end{equation*}
\end{lem}
\begin{proof}
The determinant on the left hand side is as follows:
\begin{equation}
    \det \left(I_{2ml}-M_{s_1}\cdots M_{s_k}t^{|C|}\right)
    =\sum_{\sigma \in S_{2ml}}\mathrm{sgn}(\sigma) a_{1,\sigma(1)}\cdots a_{2ml, \sigma(2ml)}, \label{eq:45}
\end{equation}
where $a_{ij}$ is the $(i,j)$-th entry of $A=I_{2ml}-M_{s_1}\cdots M_{s_k}t^{|C|}$.
Let $\mathbb{I}(i_1)=\{l(i_1-1)+1,\dots,li_1\}$.
For all $j \notin \mathbb{I}(i_1)$, 
the $j$-th row of the product $M_{s_1}\cdots M_{s_k}$ 
is a zero vector. This implies that for any $j \notin \mathbb{I}(i_1)$,
the $j$-th row of the matrix $A=I_{2ml}-M_{s_1}\cdots M_{s_k}t^{|C|}$
is identical to the $j$-th row of the identity matrix $I_{2ml}$.
Thus, the sum in \eqref{eq:45} is 
taken over permutations
$\sigma \in S_{2ml}$ that satisfy $\sigma(\mathbb{I}(i_1))=\mathbb{I}(i_1),  \ 
\sigma(j)=j$ for all $j \notin \mathbb{I}(i_1)$.
Equation \eqref{eq:45} simplifies as follows:
\begin{align*}
    \det \left(I_{2ml}-M_{s_1}\cdots M_{s_k}t^{|C|}\right)
    &=\sum_{\substack{\sigma \in S_{2ml}\\ \sigma(\mathbb{I}(i_1))=\mathbb{I}(i_1)\\\sigma(j)=j, {}^{\forall}j \notin \mathbb{I}(i_1)}}\mathrm{sgn}(\sigma) \prod_{i \in \mathbb{I}(i_1)}a_{i,\sigma(i)}\\
    &=\det \left(A_{\mathbb{I}(i_1)} \right),
\end{align*}
where $A_{\mathbb{I}(i_1)}$ is the $ l \times l$ block matrix of $A$ 
corresponding to the rows and columns indexed by $\mathbb{I}(i_1)$.

Now, we evaluate the product of block matrices 
$M_{s_1}\cdots M_{s_k}$.
Let us denote each $M_{s_u}$ as a $2m \times 2m$ block matrix
$\left(m_{\alpha, \ \beta }^{(s_u)}\right)_{1\le \alpha, \ \beta  \le 2m}$ where 
each block $m_{\alpha, \ \beta }^{(s_u)}$ is a $l \times l$ matrix.

Then the $(\alpha, \ \beta )$-th entry of the product is as follows:
\begin{equation*}
    \left(M_{s_1}\cdots M_{s_k} \right)_{\alpha, \ \beta }=
    \sum_{j_1,\dots,j_{k-1}}m_{\alpha, j_1}^{(s_1)}m_{j_1, j_2}^{(s_2)}\cdots m_{j_{k-1} ,\beta}^{(s_k)}.
\end{equation*}

By definition, 
$M_{s_u}=\left(m_{\alpha, \ \beta }^{(s_u)}\right)_{1\le \alpha, \ \beta  \le 2m}$ is as follows:
\begin{equation*}
    m_{\alpha, \ \beta }^{(s_u)}=\begin{cases}
        b_{\alpha, \ \beta }+uj_{\alpha, \ \beta }, & \text{if} \  \alpha=s_u,\\
        0, & \text{otherwise}.
    \end{cases}
\end{equation*}

Then the product simplifies as follows:
\begin{equation*}
        \left(M_{s_1}\cdots M_{s_k} \right)_{\alpha, \ \beta }= \delta_{\alpha,s_1}
    m_{\alpha,s_2}^{(s_1)}m_{s_2, s_3}^{(s_2)}\cdots m_{s_k ,\beta}^{(s_k)}.
\end{equation*}

We are interested in the block matrix corresponding to the 
indices in $\mathbb{I}(i_1)$. This is the diagonal block 
of the product whose index is $\alpha=\beta=i_1$.
For this block to be non-zero, the sequence of edges $C=(e_{i_1},e_{i_2},\dots,e_{i_k})$
must satisfy 
$t(e_{i_\alpha})=o(e_{i_{\alpha+1}}), $ for all $\alpha=1,\dots,k-1$ and $t(e_{i_{k}})=o(e_{i_1})$,
which is a condition for $C$ to be cycle.
This completes the proof.

\end{proof}

We prove the following lemmas that relate the 
Lyndon words and equivalence classes of prime cycles.

Under the same assumption as in Theorem \ref{thm11},
and by Remark \ref{rem:Lyncom},
the set of Lyndon words over $\{1,\dots,2m\}$ is a complete representatives for 
the conjugacy classes of primitive words over $\{1,\dots,2m\}$.
It follows that there is a bijection between
the set of Lyndon words and 
the equivalence classes of sequences of directed edges whose 
indices are primitive words over $\{1,\dots,2m\}$.

It is clear that any sequence of directed edges $C$, which is a cycle,
is prime if and only if 
the indices of $C$ is primitive words over $\{1,\dots,2m\}$.

\begin{lem}\label{lem:lynpri}
Let us adopt the same assumption as in Theorem \ref{thm11}.
Let $A=\{1,\dots,2m\}$ with $m=\#E(R(B_H))$.
    Let $\mathcal{C}$ be
    the set of equivalence classes of all prime cycles.
Let $L$ be the set of all Lyndon words over $A$.
For each $i \in A$,
let $M_i$ be the matrix that is defined in Lemma \ref{lem:45}.
Then the following holds:
    \begin{equation}\label{eq:lem46}
        \prod_{p \in L}\det \left(I_{2ml}-M_pt^{|p|}\right)
        =\prod_{[C] \in \mathcal{C}} 
    \det \left(I_l-\rho(\phi(C))u^{\mathrm{cbc}(C)}t^{|C|} \right),
    \end{equation}
    where the product on the left hand side runs
    over all Lyndon words $p \in L$, and for
     $ p=i_1 \cdots i_n\in L$, we define $M_p=M_{i_1}\cdots M_{i_n}$.
     The product on the right hand side runs over all equivalence classes 
     of prime cycles $[C]$, where for a cycle $C=(e_{i_1},\dots,e_{i_k})$,
     we define $M_C=M_{i_1}\cdots M_{i_k}$.
\end{lem}
\begin{proof}

By Definition \ref{def:lotha} and Lemma \ref{lem:prilo},
the set $L$ of Lyndon words over $A$ is a complete set of representatives 
for the conjugacy classes of primitive words over $A$.
Since the index word of every prime cycle is a primitive word,
and a prime cycle is equivalent to only prime cycles,
all equivalence classes of finite sequences of directed edges
are disjoint union of 
the equivalence classes of prime cycles and the other,
whose indices are conjugate as words but not cycles.

Thus the following holds:
\begin{equation}\label{eq:detLS}
    \prod_{p \in L}\det \left(I_{2ml}-M_pt^{|p|}\right)
        =\prod_{[C] \in \mathcal{C}} 
    \det \left(I_{2ml}-M_Ct^{|C|}\right)
    \cdot
    \prod_{\substack{[D]\\ D:\ \text{not cycle}}} 
    \det \left(I_{2ml}-M_Dt^{|D|}\right),
\end{equation}
where $[D]$ is the conjugacy class of finite sequences of directed edges $D$
whose index are primitive words but $D$ is not a cycle.
By Lemma \ref{lem:45}, 
the right hand side of equation \eqref{eq:detLS}
is simplified to the product over prime cycles.
Furthermore, by Lemma \ref{lem:cycleeq},
the value of the determinant for each equivalence class is 
independent of the choice of the representative.
Thus we have the following:
\begin{equation*}
    \prod_{[C] \in S} 
    \det \left(I_{2ml}-M_Ct^{|C|}\right)
    =\prod_{[C] \in \mathcal{C}}
    \det \left(I_l-\rho(\phi(C))u^{\mathrm{cbc}(C)}t^{|C|} \right).
\end{equation*}
This completes this proof.
\end{proof}

We use the following theorem for the proof.

    \begin{thm}[\cite{Amitsur}]\label{thm.amitsur}
        Let $M_1, \dots, M_k \in M_m(\ZZ)$ for some $m \in \NN$.
        % be square matrices of the same size. 
        Let $L$ be the set of all Lyndon words over $\{1,\dots,k\}$.
        Then we have the following formula:
        \begin{equation*}
            \det (I-(M_1+\cdots+M_k)t)=\prod_{p \in L}\det \left(I-M_pt^{|p|}\right),
        \end{equation*}
        where the right hand side is the product over all Lyndon words $p \in L$, and for
        $ p=i_1 \cdots i_n\in L$, we define $M_p=M_{i_1}\cdots M_{i_n}$.
    \end{thm}

We now prove the following proposition.

    \begin{prop}\label{Prop:hashimoto}
Let us adopt the same assumption as in Theorem \ref{thm11}.
Let us adopt the same assumption as Theorem \ref{thm41},
% Under the same assumption in Theorem \ref{thm41},
and let $\rho$ be a representation of $\Gamma$ with degree $l$.
Then the following formula holds:
        \begin{equation*}
            \zeta(B_H,\rho,\phi,u,t)^{-1}
            =\det \left(I-t(B+uJ)\right).
        \end{equation*}
    \end{prop}

\begin{proof}
Let  
$M_1,\dots,M_{2m}$ be $2ml \times 2ml$ matrices such that for each $i$, 
the rows indexed by $l(i-1)+1, \dots,li$ of $M_i$ 
are equal to 
corresponding rows 
of $B+uJ$,
and all other rows are zero.
It is clear that $M_1+\cdots +M_{2m}=B+uJ$.

By Lemma \ref{lem:45}, 
for any sequence of directed edges $C=(e_{i_1},\dots, e_{i_k})$, we have 
\begin{equation*}
    \det (I_{2ml}-M_{i_1}\cdots M_{i_k}t^{|C|})=
    \begin{cases}
        \det \left(I_l-\rho(\phi(C))u^{\mathrm{cbc}(C)}t^{|C|}\right),
        & \text{if} \  C \text{ is a cycle},\\
        1 & \text{otherwise.}
    \end{cases}
\end{equation*}

Applying Theorem \ref{thm.amitsur} to the matrices $M_1,\dots,M_{2m}$
and by Lemma \ref{lem:lynpri}, we obtain:
\begin{align*}
    \det(I_{2ml}-(B+uJ)t)
    &= \det (I_{2ml}-(M_1+\cdots +M_{2m})t)\\
    &=\prod_{[C]\  \mathrm{prime \ cycle}} 
    \det \left(I_l-\rho(\phi(C))u^{\mathrm{cbc}(C)}t^{|C|}\right)\\
    &=\zeta(B_H,\rho,\phi,u,t)^{-1}.
\end{align*}

\end{proof}

Under the same assumption as in Theorem \ref{thm11},
    for our convenience, we introduce the following
    $2ml \times nl$ matrices, where 
    each entry $k_{i,\beta},l_{i,\beta}$ is a $l \times l$ matrix:
    \begin{align*}
            K&=(k_{\alpha, \ \beta })_{1\le \alpha \le 2m, \ 1 \le \beta \le n},\quad
            k_{\alpha, \ \beta }=
            \begin{cases}
                \rho(\phi(e_\alpha)) & \text{if} \  t(e_\alpha)=v_\beta,\\
                0_l& \text{otherwise},
            \end{cases}\\
            L&=(l_{\alpha, \ \beta })_{1\le \alpha \le 2m, \ 1 \le \beta \le n},\quad
            l_{\alpha, \ \beta }=
            \begin{cases}
                I_l & \text{if} \ o(e_\alpha)=v_\beta,\\
                0_l& \text{otherwise}.
            \end{cases}
        \end{align*}

\begin{lem}\label{lem:47}
Under the same assumption as in Theorem \ref{thm11},
the following holds:
    \begin{equation*}
        K\ {}^tL=B+J.
    \end{equation*}
\end{lem}
\begin{proof}
It suffices to show that 
    for all $(\alpha, \ \beta  ), \ 1 \le \alpha, \ \beta   \le 2m$,
    \begin{equation}\label{eq:lem48}
        \sum_{s=1}^{n}k_{\alpha s} \ {}^tl_{\beta s}=b_{\alpha\beta }+j_{\alpha\beta }
    \end{equation}
    holds.
    Each entry $k_{\alpha s}{}^tl_{\beta s} \neq 0_l$ if and only if both
    $k_{\alpha s} \neq 0_l$ and $l_{\beta s} \neq 0_l$.
    This is because any non zero entry is an image of representation $\rho$,
    that is an 
    invertible matrix from
    $\mathrm{GL}_l(\CC)$, so the product of 
    two such entries cannot be the zero matrix.
    It is when the vertex $v_s$ is the terminus vertex of $e_\alpha $ and 
    the origin vertex of $e_\beta $.
    It is denoted by $t(e_\alpha )=v_s=o(e_\beta )$.
     This condition implies that the terminus vertex of edge $e_\alpha $ is the same as the origin vertex of edge $e_\beta $.
    If such a vertex $v_s$ exists, it is unique
    since the origin and terminus vertices of any given edge are unique.
    Thus, when there exists a vertex $v_s$ with $t(e_\alpha )=v_s=o(e_\beta )$,
    the sum of the left hand side of equation \eqref{eq:lem48}
    is not zero, and it 
    simplifies to 
    $k_{\alpha s}{}^tl_{\beta s} =\rho(\phi(e_\alpha ))I_l=\rho(\phi(e_\alpha ))$
    by the definition of $K$ and $L$.
    Therefore the left hand side of equation \eqref{eq:lem48}
    can be summarized to as follows:
    \begin{equation*}
    \sum_{s=1}^{n}k_{\alpha s}{}^tl_{\beta s}=
        \begin{cases}
            \rho(\phi(e_\alpha )), & 
            \text{if}\ t(e_\alpha )=o(e_\beta ),\\
            0_l, & \text{otherwise}.
        \end{cases}
    \end{equation*}
    Next, we consider the right hand side of equation \eqref{eq:lem48}. 
    The $(\alpha ,\beta )$-th entry of
    the sum on the right hand side of equation \eqref{eq:lem48}   
    is not zero if and only if $t(e_i)=o(e_\beta )$.
    If $t(e_\alpha )=o(e_\beta ) $ and $ e_\alpha  \neq e_\beta ^{-1}$ hold,
    $b_{\alpha \beta } \neq 0_l$ and $j_{\alpha \beta }=0_l$.
    If $t(e_\alpha )=o(e_\beta ) $ and $e_\alpha =e_\beta^{-1}$ hold, 
    $b_{\alpha \beta } = 0_l$ and $j_{\alpha \beta }\neq 0_l$.
    However, when it satisfies $t(e_\alpha )=o(e_\beta )$, 
    whether $e_\alpha =e_\beta ^{-1}$ or $e_\alpha  \neq e_\beta ^{-1}$,
    the $(\alpha ,\beta )$-th entry becomes $\rho(\phi(e_\alpha ))$.
    Therefore the right hand side of equation \eqref{eq:lem48}
    can be summarized as follows:
    \begin{equation*}
        b_{\alpha \beta }+j_{\alpha \beta }=
        \begin{cases}
            \rho(\phi(e_\alpha )), &
            \text{if}\ t(e_\alpha )=o(e_\beta ),\\
            0_l, & \text{otherwise}.
        \end{cases}
    \end{equation*}
    Thus both sides of equation \eqref{eq:lem48} are the same.
    This completes the proof.
\end{proof}

\begin{lem}\label{lem:48}
Under the same assumption as in Theorem \ref{thm11},
the following holds:
    \begin{equation*}
        {}^tLK=\sum_{g \in \Gamma}A(B_H)_g \otimes \rho(g).
    \end{equation*}
\end{lem}
\begin{proof}
    It suffices to show that 
    for all $(\alpha ,\beta ), \ 1 \le \alpha ,\beta  \le n$,
    \begin{equation}\label{eq:lem49}
        \sum_{s=1}^{2m}{}^tl_{s\alpha }k_{s\beta }
        =\sum_{g \in \Gamma}a_{\alpha \beta }^{(g)} \rho(g)
    \end{equation}
    holds.
    Each entry ${}^tl_{s\alpha }k_{s\beta } \neq 0_l$ if and only if 
    $l_{s\alpha } \neq 0_l$ and  $k_{s\beta } \neq 0_l$.
    This is because any non zero entry is an image of representation $\rho$,
    that is an 
    invertible matrix from
    $\mathrm{GL}_l(\CC)$, so the product of 
    two such entries cannot be the zero matrix.
    It is when
    $o(e_s)=v_\alpha , \ t(e_s)=v_\beta $.
    This condition implies $e_s=(v_\alpha ,v_\beta )$.
    If such an edge $e_s$ exists, it is unique 
    since the origin and terminus vertices 
    of any given directed edge are unique.
    When $e_s=(v_\alpha ,v_\beta )$ holds,
    the sum of 
    the left hand side of equation \eqref{eq:lem49} 
    simplifies to 
    ${}^tl_{s\alpha }k_{s\beta } =\rho(\phi(e_s))=\rho(\phi(v_\alpha ,v_\beta ))$.
    Therefore the left hand side of equation \eqref{eq:lem49}
    can be summarized as follows:
    \begin{equation*}
    \sum_{s=1}^{2m}{}^tl_{s\alpha }k_{s\beta }=
        \begin{cases}
            \rho(\phi(v_\alpha ,v_\beta )), & 
            \text{if}\ (v_\alpha ,v_\beta )\in E(R(B_H)),\\
            0_l, & \text{otherwise}.
        \end{cases}
    \end{equation*}
    Next, we consider the right hand side of equation \eqref{eq:lem49}.
    By definition, $\rho(g) \neq 0_l$ for all  $g \in \Gamma$.
    The condition $a_{\alpha\beta}=1$ is equivalent to the existence 
    of a unique element $g \in \Gamma$ such that $a_{\alpha \beta }^{(g)}=1$.
    Thus 
    the right hand side of equation \eqref{eq:lem49} 
    is simplified to at most one term, and 
    it is not zero if and only if 
    $a_{\alpha \beta }^{(g)} \neq 0$ for some $g \in \Gamma$.
    This happens when there exists the edge $e=(v_\alpha ,v_\beta )$ such that $\phi(e)=g$.
    Since the bipartite graph $B_H$ is simple (cf. Remark \ref{rem:bipar}),
    the edge $e=(v_\alpha ,v_\beta )$ is unique if it exists.
    If such an edge exists, the summation of
    the right hand side of equation \eqref{eq:lem49}
    simplifies to $1 \cdot \rho(g)=\rho(\phi(e))=\rho(\phi(v_\alpha ,v_\beta ))$.
    Therefore the right hand side of equation \eqref{eq:lem49}
    can be summarized as follows:
    \begin{equation*}
        \sum_{g \in \Gamma}a_{\alpha \beta }^{(g)} \rho(g)=
        \begin{cases}
            \rho(\phi(v_\alpha ,v_\beta )), & \text{if} \ (v_\alpha ,v_\beta ) \in E(R(B_H)),\\
            0_l, & \text{otherwise}.
        \end{cases}
    \end{equation*}
    Thus both sides of equation \eqref{eq:lem49} are the same.
    This completes the proof.
\end{proof}

\begin{lem}\label{lem:49}
Let us adopt the same assumption as in Theorem \ref{thm11}.
Let $\rho$ be a unitary representation of $\Gamma$ with degree $l$.
Then the following holds:
    \begin{equation*}
        {}^t\overline{K}K =D(B_H) \otimes I_l,
    \end{equation*}
    where $D(B_H)$ is the degree matrix defined in Definition \ref{def:2.16}.
\end{lem}
\begin{proof}
    It suffices to show that for all $(\alpha ,\beta ), \ 1 \le \alpha ,\beta  \le n$,
    \begin{equation}\label{eq:lem410}
        \sum_{s=1}^{2m}{}^t\overline{k_{s\alpha }}k_{s\beta }=\delta_{\alpha \beta }\deg (v_\alpha ) I_l
    \end{equation}
    holds. The entry ${}^t\overline{k_{s\alpha }}k_{s\beta } \neq 0_l$ if and only if
    $k_{s\alpha } \neq 0_l$ and $k_{s\beta } \neq 0_l$.
    This is because any non zero entry is an image of representation $\rho$,
    that is an 
    invertible matrix from
    $\mathrm{GL}_l(\CC)$, so the product of 
    two such entries cannot be the zero matrix.
    It is when 
    the terminus of the directed edge $e_s$ is $v_\alpha $ and $v_\beta $,
    which denotes 
    $t(e_s)=v_\alpha , \ t(e_s)=v_\beta $.
    Since a directed edge has a unique terminus vertex, 
    this implies that $v_\alpha =v_\beta $, which means that $\alpha =\beta $.
    The number of directed edges $e\in E(R(B_H))$ with $t(e)=v_\alpha $ is 
    $\deg(v_\alpha )$ since $R(B_H)$ is the symmetric digraph of $B_H$.
    Therefore, when $t(e_s)=v_\alpha  $ and $ \alpha =\beta $ hold,
    the sum of the left hand side of equation \eqref{eq:lem410}
    simplifies to 
    \begin{equation*}
        \sum_{\substack{s, \\ t(e_s)=v_\alpha }}
        {}^t\overline{\rho(\phi(e_s))}\rho(\phi(e_s)).
    \end{equation*}
    Each term 
    ${}^t\overline{\rho(\phi(e_s))}\rho(\phi(e_s))$
    becomes the identity matrix $I_l$ since
    $\rho$ is unitary.
    Therefore the left hand side of equation \eqref{eq:lem410}
    can be summarized as follows:
    \begin{equation*}
        \sum_{s=1}^{2m}{}^t\overline{k_{s\alpha }}k_{s\beta }=
            \delta_{\alpha \beta }\deg(v_\alpha )I_l .
    \end{equation*}
\end{proof}

\begin{lem}\label{lem:410}
Let us adopt the same assumption as in Theorem \ref{thm11}.
Let $\rho$ be a unitary representation of $\Gamma$ with degree $l$.
Then the following holds:
    \begin{equation*}
        K{}^t\overline{K}=BJ+I_{2ml}.
    \end{equation*}
\end{lem}
\begin{proof}
    It suffices to show that for all $(\alpha ,\beta ), \ 1 \le \alpha ,\beta  \le 2m$,
    \begin{equation}\label{eq:lem411}
        \sum_{u=1}^n k_{\alpha u}{}^t\overline{k_{\beta u}}= 
        \sum_{v=1}^{2m}b_{\alpha v}j_{v\beta }+\delta_{\alpha \beta }I_l
    \end{equation}
    holds.
    The entry $k_{\alpha u}{}^t\overline{k_{\beta u}} \neq 0_l$ if and only if 
    $k_{\alpha u} \neq 0_l$ and $k_{\beta u} \neq 0_l$.
    This is because any non zero entry is an image of representation $\rho$,
    that is an 
    invertible matrix from
    $\mathrm{GL}_l(\CC)$, so the product of 
    two such entries cannot be the zero matrix.
    It is when 
    $t(e_\alpha )=v_u=t(e_\beta )$.
    Assume that the left hand side of equation \eqref{eq:lem411}
    is non-zero.
    Since the terminus vertex of any given directed edge is unique,
    the sum of the left hand side simplifies to a
    single term.
    If $\alpha \neq \beta $, the term is $\rho(\phi(e_\alpha )){}^t \overline{\rho(\phi(e_\beta ))}$.
    If $\alpha =\beta $, the term is $\rho(\phi(e_\alpha )){}^t \overline{\rho(\phi(e_\alpha ))}$.
    However, since $\rho$ is unitary, 
    then it turns out to be the identity matrix $I_l$.
    Therefore, the left hand side of equation \eqref{eq:lem411}
    can be summarized as follows:
    \begin{equation*}
        \sum_{u=1}^n k_{\alpha u}{}^t\overline{k_{\beta u}}= 
        \begin{cases}
            \rho(\phi(e_\alpha )){}^t \overline{\rho(\phi(e_\beta ))}, &
            \text{if}\ \alpha \neq \beta , \ t(e_\alpha )=t(e_\beta ),\\
            I_l, & \text{if}\  \alpha =\beta ,\\
            0_l, & \text{otherwise}.
        \end{cases}
    \end{equation*}
    Next we consider the right hand side of equation \eqref{eq:lem411}.
    The entry $b_{\alpha v}j_{v\beta } \neq 0$ if and only if $t(e_\alpha )=o(e_v),\ e_\alpha  \neq e_{v}^{-1}, \ e_v=e_{\beta }^{-1}$.
    Since the inverse of a directed edge is unique,
    if the edge $e_v$ that satisfies $t(e_\alpha )=o(e_v),\ e_\alpha  \neq e_{v}^{-1}, \ e_v=e_{\beta }^{-1}$ for some $(\alpha ,\beta )$ exists then it is unique.
    If such an edge $e_v$ exists, then $(\alpha ,\beta )$-th entry of 
    $BJ$ on the right hand side of equation \eqref{eq:lem411}
    simplifies as follows:
    \begin{align*}
        \rho(\phi(e_\alpha ))\rho(\phi(e_v))
        &=\rho(\phi(e_\alpha ))\rho(\phi(e_\beta ^{-1}))\\
        &=\rho(\phi(e_\alpha )){}^t \overline{\rho(\phi(e_\beta ))},
    \end{align*}
    since $\rho$ is unitary.
    Therefore the right hand side of equation \eqref{eq:lem411}
    can be summarized as follows:
    \begin{equation*}
        \sum_{v=1}^{2m}b_{\alpha v}j_{v\beta }+\delta_{\alpha \beta }I_l=
        \begin{cases}
            \rho(\phi(e_\alpha )){}^t \overline{\rho(\phi(e_\beta ))}
            &\text{if}\ 
            t(e_\alpha )=t(e_\beta ), \ \alpha  \neq \beta ,\\
            I_l, & \text{if}\ \alpha =\beta ,\\
            0_l, & \text{otherwise}.
        \end{cases}
    \end{equation*}
    Thus both sides of equation \eqref{eq:lem411} are the same.
\end{proof}

\begin{lem}\label{lem:411}
Under the same assumption as in Theorem \ref{thm11}, the following holds:
    \begin{equation*}
        J^2=I_{2ml}.
    \end{equation*}
\end{lem}
\begin{proof}
It suffices to show that for all $(\alpha ,\beta ), \ 1 \le \alpha ,\beta  \le 2m$,
\begin{equation}\label{eq:lem412}
    \sum_{s=1}^{2m}j_{\alpha s}j_{s\beta }=I_l
\end{equation}
holds.
    Each entry $j_{\alpha s}j_{s\beta } \neq 0$ if and only if 
    $j_{\alpha s} \neq 0_l$ and $j_{s\beta } \neq 0_l$.
    This is because any non zero entry is an image of representation $\rho$,
    that is an 
    invertible matrix from
    $\mathrm{GL}_l(\CC)$, so the product of 
    two such entries cannot be the zero matrix.
    It is when 
    $e_\alpha =e_s^{-1}=e_\beta $. 
    This condition holds when $e_s=e_\alpha ^{-1}$ and $\alpha =\beta $ since the inverse of 
    any given directed edge is unique. Thus the summation of the left 
    hand side of equation \eqref{eq:lem412} simplifies to a single term
    and it becomes as follows;
    \begin{align*}
        j_{\alpha s}j_{s\beta } 
        &=\rho(\phi(e_\alpha ))\rho(\phi(e_s))\\
        &=\rho(\phi(e_\alpha ))\rho(\phi(e_\alpha ^{-1}))\\
        &=I_l.
    \end{align*}
    Therefore, the left hand side of equation \eqref{eq:lem412}
    can be summarized as follows:
    \begin{equation*}
        \sum_{s=1}^{2m}j_{\alpha s}j_{s\beta }=
        \begin{cases}
            I_l, & \text{if}\ \alpha =\beta ,\\
            0_l, & \text{otherwise}.
        \end{cases}
    \end{equation*}
    Thus, both sides of equation \eqref{eq:lem412} are the same.
\end{proof}

We present some lemmas that are useful for the proof of 
Proposition \ref{prop54}.
 
\begin{lem}\label{lem:412}
Under the same assumption as in Theorem \ref{thm11}, the following holds:
    \begin{equation*}
        \det (I_{2ml}-(1-u)tJ) 
        =(1-(1-u)^2t^2)^{ml}.
    \end{equation*}
\end{lem}
\begin{proof}
    We adopt the indexing of directed edges introduced in Remark \ref{rem:24},
    We first define a $2m \times 2m$ block matrix $P=(p_{\alpha \beta })$, where 
    each block $p_{\alpha \beta }$ is an $l \times l$ matrix.
    The entries of $P$ are given as follows:
    \begin{equation*}
    p_{\alpha \beta }=
    \begin{cases}
        \rho(\phi(e_\alpha )), & 
        \text{if}\ \alpha <\beta , \   e_\alpha =e_\beta ^{-1},\\
        0_l & \text{otherwise}.
    \end{cases}
    \end{equation*}
    Using the index of Remark \ref{rem:24},
    the entries can be written as follows:
    \begin{equation*}
    p_{\alpha \beta }=
    \begin{cases}
        \rho(\phi(e_\alpha )), & \text{if}\ 
        1 \le \alpha  \le m,\  \beta =\alpha +m, \\
        0_l & \text{otherwise}.
    \end{cases}
    \end{equation*}
    The matrix $P$ is a strictly upper triangular matrix, which implies that 
    $\det (I_{2ml}+(1-u)tP)=1$.
    
    We begin with the difference $P-J$.
    By the definitions of $P $ and $J$,
    both matrices have the same entry, $\rho(\phi(e_\alpha ))$,
    in the $(\alpha ,\alpha +m)$-th entry for $\alpha \in \{1,\dots,m\}$.
    Thus, these common entries cancel each other out.
    The $(\alpha +m,\alpha )$-th entry of $J$ 
    for $\alpha \in \{1,\dots,m\}$ remains,
    as the corresponding entries in $P$ are zero.
    The result is as follows:
    \begin{equation}\label{eq:p-j}
        (P-J)_{\alpha \beta }=
        \begin{cases}
            -\rho(\phi(e_\alpha )) & 
            \text{if}\ m+1 \le \alpha  \le 2m, \beta =\alpha -m,\\
            0_l & \text{otherwise}.
        \end{cases}
    \end{equation}
    Thus, $P-J$ is a strictly lower triangular matrix.
    
    Next, we compute the product $PJ$.
    The $(\alpha ,\beta )$-th entry of this product is given by the sum
    $\sum_{s=1}^{2m}p_{\alpha s}j_{s\beta }$. By the definitions of $P$ and $J$,
    each entry $p_{\alpha s}j_{s\beta }\neq 0$ if and only if
    $e_\alpha =e_s^{-1}=e_\beta $ for $\alpha  \in \{ 1,\dots, m\}$ and $s=\alpha +m$.
    Since the inverse of any directed edge is unique, this condition 
    implies that 
    $e_\alpha =e_s^{-1}$ for $\alpha  \in \{ 1,\dots, m\}, \ s=\alpha +m$ and $\alpha=\beta $. 
    Then for all $\alpha  \in \{1,\dots, m\}$,
    \begin{equation*}
    \sum_{s=1}^{2m}p_{\alpha s}j_{s\beta }
    =\delta_{\alpha \beta }\rho(\phi(e_\alpha ))\rho(\phi(e_s))
    =\delta_{\alpha \beta }\rho(\phi(e_\alpha ))\rho(\phi(e_\alpha ^{-1}))
    =\delta_{\alpha \beta }I_l.
    \end{equation*}
    Thus
    \begin{equation}\label{eq:pj}
        (PJ)_{\alpha \beta }=
        \begin{cases}
            \delta_{\alpha \beta }I_l, &
            \text{if}\ \alpha  \in \{1,\dots, m\},\\
            0_l, & \text{otherwise},
        \end{cases}
    \end{equation}
    which shows that the matrix $PJ$ is a diagonal matrix.

    We calculate the determinant of $(I_{2ml}-(1-u)tJ)$
    as follows:
    \begin{align*}
        1 \cdot  \det (I_{2ml}-(1-u)tJ) 
        &=\det (I_{2ml}+(1-u)tP)\cdot \det (I_{2ml}-(1-u)tJ)\\
        &=\det(I_{2ml}+(1-u)t(P-J)-(1-u)^2t^2 PJ)
    \end{align*}
    Let $R=I_{2ml}+(1-u)t(P-J)-(1-u)^2t^2 PJ$. Then
    by equations \eqref{eq:p-j} and \eqref{eq:pj},
    $R$ is a lower triangular matrix.
    Thus the determinant of $R$ 
    is the product of its diagonal entries.
    The diagonal entries of $R$ are as follows:
    \begin{equation*}
        r_{\alpha \alpha }
        =\begin{cases}
            (1-(1-u)^2t^2)I_l, & 
            \text{if}\ 1 \le \alpha  \le m,\\
            I_l, &
            \text{if}\ m+1 \le \alpha  \le 2m, \\
        \end{cases}
    \end{equation*}
    Thus we obtain the following formula:
    \begin{equation*}
        \det (I_{2ml}-(1-u)tJ) =(1-(1-u)^2t^2)^{ml}.
    \end{equation*}
\end{proof}

\begin{lem}\label{lem:change}
    Let $A\in M_s(\ZZ)$ and $B \in M_t(\ZZ)$.
    Then, there exists a matrix $\mathfrak{P}_{s,t} \in GL_{st}(\ZZ)$
    such that $\mathfrak{P}_{st}(A \otimes B)\mathfrak{P}_{st}^{-1}=B \otimes A$.
    Here $A \otimes B$ is the Kronecker product in Definition \ref{def:kronecker}.
\end{lem}
\begin{proof}
The matrices $A \in M_s(\ZZ)$ and $B \in M_t(\ZZ)$ represent linear maps $f_A$ 
and $f_B$ acting on vector spaces $V$ and $W$ with 
respect to the standard bases 
$\{e_i\}_{1 \le i \le s}$ and $\{f_j\}_{1 \le j \le t}$, respectively.
The Kronecker product $A \otimes B$ is the transformation matrix 
of the linear map $f_A \otimes f_B: V \otimes W \to V \otimes W$ 
with respect to the standard basis $B_1$, which is ordered as:
$B_1=(e_1 \otimes f_1,e_1\otimes f_2,\dots, e_1 \otimes f_t,\dots,e_s \otimes f_t )$.
Similarly,
the matrix $B \otimes A$ is the transformation matrix 
of the linear map $f_B \otimes f_A: W \otimes V \to W \otimes V$ 
with respect to the standard basis $B_2$, which is ordered as:
$B_2=(f_1 \otimes e_1,f_2 \otimes 
e_1,\dots,f_t \otimes e_1,\dots, f_t \otimes e_s )$.
Let $f_{\mathfrak{P}_{st}}:V \otimes W \to W \otimes V$ be the linear map 
that swaps the tensor factors $f_{\mathfrak{P}_{st}}(e_i \otimes f_j)=f_j \otimes e_i$.
We define the matrix $\mathfrak{P}_{st}$ as the transformation matrix 
with respect to $f_{\mathfrak{P}_{st}}$.
Since $\mathfrak{P}_{st}$ is a permutation matrix, it is invertible,
and thus $\mathfrak{P}_{st} \in GL_{st}(\ZZ)$.
By the definition of $\mathfrak{P}_{st}$, we have the equation
$f_{B \otimes A}=f_{\mathfrak{P}_{st}}f_{A \otimes B} 
f_{\mathfrak{P}_{st}}^{-1}$.
Therefore, $B \otimes A=\mathfrak{P}_{st}(A \otimes B)\mathfrak{P}_{st}^{-1}$ holds.

\end{proof}

\begin{lem}\label{lem:413}
Under the same assumption as in Theorem \ref{thm11}, the following holds:

    \begin{align*}
         &\det
        \left(I_{nl}-t\sum_{g \in \Gamma}A(B_H)_g \otimes \rho(g)
            +(1-u)t^2 (D(B_H) \otimes I_l-(1-u)I_{nl}) \right)\\
        &= \det
        \left(I_{nl}-t\sum_{g \in \Gamma}\rho(g) \otimes A(B_H)_g
            +(1-u)t^2 (I_l \otimes D(B_H)-(1-u)I_{nl}) \right).
    \end{align*}
\end{lem}
\begin{proof}
    By Lemma \ref{lem:change}, there exists a permutation matrix $\mathfrak{P}_{nl}$
    such that 
    $\mathfrak{P}_{nl}(A(B_H)_g \otimes \rho(g))\mathfrak{P}_{nl}^{-1}=\rho(g) \otimes A(B_H)_g$
    and 
    $\mathfrak{P}_{nl}(D(B_H) \otimes I_l)\mathfrak{P}_{nl}^{-1}=I_l \otimes D(B_H)$ 
    hold.
    Then
    \begin{align*}
        &\det
        \left(I_{nl}-t\sum_{g \in \Gamma}A(B_H)_g \otimes \rho(g)
            +(1-u)t^2 (D(B_H) \otimes I_l-(1-u)I_{nl}) \right)\\
        &=\det(\mathfrak{P}_{nl})
            \det
        \left(I_{nl}-t\sum_{g \in \Gamma}A(B_H)_g \otimes \rho(g)
            +(1-u)t^2 (D(B_H) \otimes I_l-(1-u)I_{nl}) \right)
            \det(\mathfrak{P}_{nl}^{-1})\\
        &=\det \left( 
        \mathfrak{P}_{nl} \left(I_{nl}-t\sum_{g \in \Gamma}A(B_H)_g \otimes \rho(g)
            +(1-u)t^2 (D(B_H) \otimes I_l-(1-u)I_{nl}) \right)\mathfrak{P}_{nl}^{-1}\right)\\
        &=\det \left(I_{nl}-t\sum_{g \in \Gamma}\rho(g) \otimes A(B_H)_g
            +(1-u)t^2 (I_l \otimes D(B_H)-(1-u)I_{nl}) \right).
    \end{align*}
    This completes the proof.
\end{proof}

    \begin{prop}\label{prop54}

Let us adopt the same assumption as in Theorem \ref{thm11}.
Let $\rho$ be a unitary representation of $\Gamma$ with degree $l$.
Then the following formula holds:
        \begin{align*}
            \zeta(B_H,\rho,\phi,u,t)^{-1}
            &=(1-(1-u)^2t^2)^{(m-n)l}\\
            &\cdot \det \left(I_{nl}-t\sum_{g \in \Gamma}\rho(g) \otimes A(B_H)_g+
            (1-u)t^2\left(I_l \otimes (D(B_H)-(1-u)I_n) \right)\right).
        \end{align*}
    \end{prop}

\begin{proof}
    By Proposition \ref{Prop:hashimoto},
    it suffices to show that the right hand side
    is equal to $\det \left(I-(B+uJ)t\right)$.\\
    By Lemma \ref{lem:47}, Lemma \ref{lem:48},
    Lemma \ref{lem:49}, Lemma \ref{lem:410} and Lemma \ref{lem:411}, we have
    \begin{align*}
        K\ {}^tL&=B+J,\quad
        {}^tLK=\sum_{g \in \Gamma}A(B_H)_g \otimes \rho(g),\quad
        % \intertext{holds. Also,}
        {}^t\overline{K}K =D(B_H) \otimes I_l,\quad
        % \intertext{holds. The following }
        K\ {}^t\overline{K}=BJ+I_{2ml},\quad
        J^2=I_{2ml}.
        \\
        \intertext{We now define the following matrices:}
        \mathbb{X}&=
        \begin{pmatrix}
            (1-(1-u)^2t^2)I_{nl} & -{}^tL+(1-u)t{}^t\overline{K} \\
            0 & I_{2ml} 
        \end{pmatrix},\quad 
        \mathbb{Y}=
        \begin{pmatrix}
            I_{nl} & {}^tL-(1-u)t{}^t\overline{K} \\
            tK & (1-(1-u)^2t^2)I_{2ml}
        \end{pmatrix}.
        \intertext{Their products are:}
        \mathbb{X}\mathbb{Y}
        &=\begin{pmatrix}
            (1-(1-u)^2t^2)I_{nl}-t\sum_{g \in \Gamma}A(B_H)_g \otimes \rho(g)
            +(1-u)t^2 D(B_H) \otimes I_l & 0\\
            tK & (1-(1-u)^2t^2)I_{2ml}
        \end{pmatrix},\\
        \mathbb{Y}\mathbb{X}
        &=\begin{pmatrix}
            (1-(1-u)^2t^2)I_{nl} & 0\\
            t(1-(1-u)^2t^2)K & (I_{2ml}-t(B+uJ))(I_{2ml}-(1-u)tJ)
        \end{pmatrix}.\\
        \end{align*}
        
        The determinants of $\mathbb{X}\mathbb{Y}$ and $\mathbb{Y}\mathbb{X}$ 
        are equal. Therefore, we have the following:
        \begin{align}
        &(1-(1-u)^2t^2)^{2ml}
        \det \left(I_{nl}-t\sum_{g \in \Gamma}A(B_H)_g \otimes \rho(g)
            +(1-u)t^2 (D(B_H) \otimes I_l-(1-u)I_{nl}) \right) \label{eq:det2}  \\ 
        &=(1-(1-u)^2t^2)^{nl}
        \det \left(I_{2ml}-t(B+uJ) \right)
        \det \left(I_{2ml}-(1-u)tJ\right).\notag
        \end{align}
        By Lemma \ref{lem:412}, $\det (I_{2ml}-(1-u)tJ)=(1-(1-u)^2t^2)^{ml}$.
        Thus, we now obtain $\det  (I_{2ml}-t(B+uJ))$ by 
        substituting Lemma \ref{lem:412} 
        into equation \eqref{eq:det2}.
        
        \begin{align*}
        &\det (I_{2ml}-t(B+uJ))\\
        &=(1-(1-u)^2t^2)^{(m-n)l} \det
        \left(I_{nl}-t\sum_{g \in \Gamma}A(B_H)_g \otimes \rho(g)
            +(1-u)t^2 (D(B_H) \otimes I_l-(1-u)I_{nl}) \right)\\
            \intertext{By Lemma \ref{lem:413}, the  Kronecker product factors are interchanged:}
        &=(1-(1-u)^2t^2)^{(m-n)l} \det
        \left(I_{nl}-t\sum_{g \in \Gamma}\rho(g) \otimes A(B_H)_g
            +(1-u)t^2 (I_l \otimes D(B_H)-(1-u)I_{nl}) \right)\\
        &=(1-(1-u)^2t^2)^{(m-n)l} \det
        \left(I_{nl}-t\sum_{g \in \Gamma}\rho(g) \otimes A(B_H)_g
            +(1-u)t^2 I_l \otimes ( D(B_H)-(1-u)I_{n}) \right)
    \end{align*}    
    This concludes the proof.
\end{proof}

\begin{rem}\label{rem:412}

Proposition \ref{prop54} also applies when $\rho=\mathbb{P}$,
because the (left) permutation representation is unitary.
\end{rem}

The following Theorem holds.

    \begin{thm}[Theorem \ref{thm12}]\label{thm53}
Let us adopt the same assumption as Theorem \ref{thm11},
and let us assume that for any $i$ with $m_i>0$, $\rho_i$ is unitary.
Then the following identity holds:
        \begin{equation*}
            \zeta(\Bar{H},u,t)
            =\prod_{i=1}^s\zeta(H,\rho_i,\phi,u,t)^{m_i}.
        \end{equation*}
    \end{thm}

\begin{proof}

By Theorem \ref{thm41},
we have the following equation:
        \begin{equation*}
            \zeta(\Bar{H},u,t)^{-1}=
            \zeta(H,u,t)^{-m_1}
            (1-(1-u)^2t)^{(k-m_1)(m-n)}
            \prod_{i=2}^s M_i^{m_i},
        \end{equation*}
        where 
        \begin{equation*}
         M_i=\det \left(I_{f_in}-\sqrt{t}\sum_{g \in \Gamma}(\rho_i(g) \otimes A(B_H)_g)+
         (1-u)tf_i \circ \left(D(B_H)-(1-u)I_n\right)\right).
        \end{equation*}
Since $\rho_i$ is unitary for any $i$ with $m_i>0$,
applying Proposition \ref{prop54}, we get 
\begin{equation*}
    M_i=\zeta(B_H,\rho_i,\phi,u,\sqrt{t})^{-1}(1-(1-u)^2t)^{-(m-n)f_i}.
\end{equation*}
Comparing the sizes of the
matrices on both sides and using Lemma \ref{lem34}, we obtain the equality  
$k=\sum_{i=1}^s m_if_i$.
Since $\rho_1$ is the trivial representation, its degree $f_1=1$.
Also, by Definition \ref{def:L-funct.}, we have the identity 
$\zeta(B_H,\rho_1,\phi,u,t)=\zeta(B_H,u,t)$ .
Thus 
\begin{align*}
    \zeta(\Bar{H},u,t)^{-1}
    =&\zeta(H,u,t)^{-m_1}
            (1-(1-u)^2t)^{(k-m_1)(m-n)} \\
            &\cdot
            \prod_{i=2}^s 
            \left(\zeta(B_H,\rho_i,\phi,u,\sqrt{t})^{-1}(1-(1-u)^2t)^{-(m-n)f_i} \right)^{m_i}\\
    =&\zeta(B_H,\rho_1,\phi,u,\sqrt{t})^{-m_1}
    (1-(1-u)^2t)^{(f_2m_2+\cdots +f_sm_s)(m-n)} \\
    &\cdot 
    \prod_{i=2}^s 
            \zeta(B_H,\rho_i,\phi,u,\sqrt{t})^{-m_i}
            (1-(1-u)^2t)^{-(m-n)f_im_i}\\
    =&\prod_{i=1}^s \zeta(B_H,\rho_i,\phi,u,\sqrt{t})^{-m_i}\\
    =&\prod_{i=1}^s \zeta(H,\rho_i,\phi,u,t)^{-m_i}
\end{align*}
This completes the proof.
\end{proof}

Finally, we illustrate an example of Theorem \ref{thm53}.

\begin{ex}
    Let $H=(V(H),E(H))$ be a hypergraph with
    the vertex set $V(H)=\{v_1,v_2,v_3\}$ and
    the edge set $E(H)=\{e_1,e_2,e_3\}$, 
    where $e_1=\{v_1,v_2\},e_2=\{v_2,v_3\},e_3=\{v_1,v_2,v_3\}$.
    Let the hypergraph covering $\Bar{H}$ be defined 
    by the permutation voltage assignment $\phi$:
    $\phi((v_1,e_1))=\phi((v_1,e_3))=(12)$,
    and $\phi((v,e))=\phi((e,v))=1$ for all other
    directed edges $(v,e),(e,v) \in E(R(B_H))$.
    Then $\Gamma=S_2$, 
    and $\Bar{H}$ is the $2$-fold hypergraph covering of $H$ 
    with $n=\#V(B_{\Bar{H}})=12 
     $ and $m=\#E(V_{\Bar{H}})=14$.
    By Theorem \ref{thm213}, the reciprocal of the Bartholdi zeta function is 
    given as follows:
    \begin{equation*}
        \zeta(\Bar{H},u,t)^{-1}
        =\zeta(B_{\Bar{H}},u,\sqrt{t})^{-1}
        =(1-(1-u)^2t)^{2}\det (I_{12}-\sqrt{t}A(B_{\Bar{H}})+(1-u)t(D(B_{\Bar{H}})-(1-u)I_{12})),
    \end{equation*}
    where 
\begin{align*}
A(B_{\Bar{H}})&=
\left(
\begin{array}{cccccc|cccccc}
0 & 0 & 0 & 0 & 0 & 0 & 0 & 0 & 1 & 1 & 0 & 0 \\ 
0 & 0 & 0 & 0 & 0 & 0 & 1 & 1 & 1 & 0 & 0 & 0 \\
0 & 0 & 0 & 0 & 0 & 0 & 0 & 1 & 1 & 0 & 0 & 0 \\ 
0 & 0 & 0 & 0 & 0 & 0 & 1 & 0 & 0 & 0 & 0 & 1 \\
0 & 0 & 0 & 0 & 0 & 0 & 0 & 0 & 0 & 1 & 1 & 1 \\
0 & 0 & 0 & 0 & 0 & 0 & 0 & 0 & 0 & 0 & 1 & 1 \\ \hline
0 & 1 & 0 & 1 & 0 & 0 & 0 & 0 & 0 & 0 & 0 & 0 \\ 
0 & 1 & 1 & 0 & 0 & 0 & 0 & 0 & 0 & 0 & 0 & 0 \\ 
1 & 1 & 1 & 0 & 0 & 0 & 0 & 0 & 0 & 0 & 0 & 0 \\ 
1 & 0 & 0 & 0 & 1 & 0 & 0 & 0 & 0 & 0 & 0 & 0 \\
0 & 0 & 0 & 0 & 1 & 1 & 0 & 0 & 0 & 0 & 0 & 0 \\
0 & 0 & 0 & 1 & 1 & 1 & 0 & 0 & 0 & 0 & 0 & 0 \\ 
\end{array}
\right), \\
D(B_{\Bar{H}})&=
\mathrm{diag}(2,3,2,2,3,2,2,2,3,2,2,3).
\end{align*}
We set the basis of above two matrices as 
$\left(v_1^{(1)},v_2^{(1)},v_3^{(1)},v_1^{(2)},v_2^{(2)},v_3^{(2)}
,e_1^{(1)},e_2^{(1)},e_3^{(1)},e_1^{(2)},e_2^{(2)},e_3^{(2)}\right)$.

The group $\Gamma=S_2$ has two irreducible representations:
the trivial representation $\rho_1$ and the sign representation 
$\rho_2$. The permutation representation $\mathbb{P}$ 
is conjugate to the direct sum $\rho_1 \oplus \rho_2$,
where both irreducible representations have multiplicity $m_i=1$ and 
degree $f_i=\deg \rho_i=1$.
The $L$-functions are expressed as follows:
\begin{align*}
    \zeta(H,\rho_1,\phi,u,t)^{-1}
    &=\zeta(B_H,\rho_1,\phi,u,\sqrt{t})^{-1}\\
    &=(1-(1-u)^2t)\det (I_{6}-\sqrt{t}A(B_{H})+(1-u)tD(B_H)-(1-u)I_{6}))^{-1},\\
        \zeta(H,\rho_2,\phi,u,t)^{-1}
    &=\zeta(B_H,\rho_2,\phi,u,\sqrt{t})^{-1}\\
    &=(1-(1-u)^2t)\det 
    \left(I_{6}-\sqrt{t}\sum_{g \in S_2}\rho_2(g) \otimes A(B_{H})_g+(1-u)tD(B_H)-(1-u)I_{6})\right),
\end{align*}
where 
\begin{align*}
A(B_{H})&=
\left(
\begin{array}{ccc|ccc}
0 & 0 & 0 & 1 & 0 & 1  \\ 
0 & 0 & 0 & 1 & 1 & 1  \\
0 & 0 & 0 & 0 & 1 & 1  \\ \hline
1 & 1 & 0 & 0 & 0 & 0  \\
0 & 1 & 1 & 0 & 0 & 0  \\
1 & 1 & 1 & 0 & 0 & 0  \\ 
\end{array}
\right), \\
\sum_{g \in S_2}\rho_2(g) \otimes A(B_{H})_g
&=
\left(
\begin{array}{ccc|ccc}
0 & 0 & 0 & -1 & 0 & 1  \\ 
0 & 0 & 0 & 1 & 1 & 1  \\
0 & 0 & 0 & 0 & 1 & 1  \\ \hline
-1 & 1 & 0 & 0 & 0 & 0  \\
0 & 1 & 1 & 0 & 0 & 0  \\
1 & 1 & 1 & 0 & 0 & 0  \\ 
\end{array}
\right), \\
D(B_{H})&=
\mathrm{diag}(2,3,2,2,2,3).
\end{align*}

Thus
\begin{align*}
\zeta(\Bar{H},u,t)^{-1}
&=(1-(1-u)^2t)^2(ut - t - 1)(ut - t + 1)(u^2t^2 - t^2 - t - 1)(u^2t^2 - t^2 + t - 1) \\
& \quad \times (u^3t^3 + 2u^2t^3 - u^2t^2 - ut^3 - ut^2 - 2t^3 - ut + t^2 - t + 1) \\
& \quad \times (u^3t^3 + 2u^2t^3 + u^2t^2 - ut^3 + ut^2 - 2t^3 - ut - t^2 - t - 1) \\
& \quad \times (u^6t^6 + u^5t^6 - 4u^4t^6 - u^4t^5 - 2u^3t^6 - 3u^4t^4 + 5u^2t^6 - 2u^3t^4 + 2u^2t^5 + ut^6 \\
& \qquad + 5u^2t^4 - 2t^6 + 2u^2t^3 + ut^4 - t^5 + 3u^2t^2 - t^4 + ut^2 - t^3 - t^2 - t - 1) \\
& \quad \times (u^6t^6 + u^5t^6 - 4u^4t^6 + u^4t^5 - 2u^3t^6 - 3u^4t^4 + 5u^2t^6 - 2u^3t^4 - 2u^2t^5 + ut^6 \\
& \qquad + 5u^2t^4 - 2t^6 - 2u^2t^3 + ut^4 + t^5 + 3u^2t^2 - t^4 + ut^2 + t^3 - t^2 + t - 1),\\
\zeta(H,\rho_1,\phi,u,t)^{-1}
&=(1-(1-u)^2t)(ut - t - 1)(ut - t + 1) \\
& \quad \times (u^2t^2 - t^2 - t - 1)(u^2t^2 - t^2 + t - 1) \\
& \quad \times (u^3t^3 + 2u^2t^3 - u^2t^2 - ut^3 - ut^2 - 2t^3 - ut + t^2 - t + 1) \\
& \quad \times (u^3t^3 + 2u^2t^3 + u^2t^2 - ut^3 + ut^2 - 2t^3 - ut - t^2 - t - 1),
\end{align*}

\begin{align*}
\zeta(H,\rho_2,\phi,u,t)^{-1}
&=(1-(1-u)^2t) \\
& \quad \times (u^6t^6 + u^5t^6 - 4u^4t^6 - u^4t^5 - 2u^3t^6 - 3u^4t^4 + 5u^2t^6 - 2u^3t^4 + 2u^2t^5 + ut^6 \\
& \quad + 5u^2t^4 - 2t^6 + 2u^2t^3 + ut^4 - t^5 + 3u^2t^2 - t^4 + ut^2 - t^3 - t^2 - t - 1) \\
& \quad \times (u^6t^6 + u^5t^6 - 4u^4t^6 + u^4t^5 - 2u^3t^6 - 3u^4t^4 + 5u^2t^6 - 2u^3t^4 - 2u^2t^5 + ut^6 \\
& \quad + 5u^2t^4 - 2t^6 - 2u^2t^3 + ut^4 + t^5 + 3u^2t^2 - t^4 + ut^2 + t^3 - t^2 + t - 1).
\end{align*}

By performing the calculation,
we confirm that Theorem \ref{thm53} holds for this example.

\end{ex}

\section*{Acknowledgments}
The author would like to thank deeply Professor Hidekazu Furusho for his
helpful comments and many valuable advices.
The author is indebted to Professor Iwao Sato 
for his comments in revising this paper.

\end{document}